\documentclass[11pt,a4paper]{article}
\usepackage[utf8]{inputenc}

\usepackage[all]{xy}
\usepackage{latexsym}
\usepackage[margin=3cm]{geometry}
\usepackage[dvipdfmx]{graphicx}

\usepackage{hyperref}
\usepackage{mymacro}
\usepackage{tikz}
\usepackage{mathtools}

\usepackage[
backend=biber,
style=ext-alphabetic,
maxnames = 10,
maxalphanames = 10,
maxcitenames  = 10,
mincitenames  = 4,
minalphanames = 4,
minnames      = 1,
sorting=nyt,
giveninits,
articlein=false,
url=false,  
doi=false,
eprint=true,
isbn=false
]{biblatex}
\renewcommand{\labelenumi}{$\mathrm{(\roman{enumi})}$}
\renewcommand{\labelenumii}{$\mathrm{(\alph{enumii})}$}

\SelectTips{cm}{}

\makeatletter

\@addtoreset{equation}{section}
\makeatother

\title{Completeness of derived interleaving distances and \\ sheaf quantization of non-smooth objects
}
\author{Tomohiro Asano \and Yuichi Ike}
\date{\today}

\begin{document}
\maketitle

\begin{abstract}
    We develop sheaf-theoretic methods to deal with non-smooth objects in symplectic geometry. 
    We show the completeness of a derived category of sheaves with respect to the interleaving distance and construct a sheaf quantization of a Hamiltonian homeomorphism.
    We also develop Lusternik--Schnirelmann theory in the microlocal theory of sheaves. 
    With these new sheaf-theoretic methods, we prove an Arnold-type theorem for the image of a compact exact Lagrangian submanifold under a Hamiltonian homeomorphism.
\end{abstract}

\section{Introduction}

The microlocal theory of sheaves due to Kashiwara and Schapira~\cite{KS90} has been effectively applied to symplectic and contact geometry after the pioneering work by Nadler--Zaslow~\cite{NZ,Nad} and Tamarkin~\cite{Tamarkin}. 
In this paper, we apply the theory to symplectic geometry for non-smooth objects, in particular, limits of smooth objects.

\subsection{Our results}

In this work, we use the Tamarkin category, which was introduced by Tamarkin~\cite{Tamarkin} to prove his non-displaceability theorem. 
Let $X$ be a connected $C^\infty$-manifold without boundary and let $\pi \colon T^*X \to X$ denote its cotangent bundle equipped with the canonical symplectic structure. 
We also let $(t;\tau)$ denote the homogeneous coordinate system on $T^*\bR_t$ and fix a field $\bfk$. 
The Tamarkin category $\cD(X)$ is defined to be the left orthogonal of the triangulated subcategory consisting of objects microsupported in $\{ \tau \le 0\} \subset T^*(X \times \bR_t)$ in $\SD(\bfk_{X \times \bR_t})$, the derived category of sheaves on $X \times \bR_t$. 
To generalize Tamarkin's non-displaceability theorem to a quantitative version, the authors introduced the interleaving-like pseudo-distance on $\cD(X)$ in \cite{AI20}, motivated by the convolution distance on $\SD(\bfk_{\bR^n})$ due to Kashiwara--Schapira~\cite{KS18persistent}.
In this paper, we introduce a modified pseudo-distance $d_{\cD(X)}$ on $\cD(X)$ that is possibly larger than the previous one.
See \cref{subsec:tamarkin} for the details of the Tamarkin category and the pseudo-distance $d_{\cD(X)}$.

In this paper, we first establish the completeness of the Tamarkin category $\cD(X)$ with respect to $d_{\cD(X)}$ in \cref{section:limit}.

\begin{theorem}[{see \cref{corollary:limit_Tamarkin}}]\label{theorem:complete_intro}
	The Tamarkin category $\cD(X)$ is complete with respect to the pseudo-distance $d_{\cD(X)}$.
\end{theorem}

In fact, we prove the completeness result for a wider class of distances associated with thickening kernels, which was introduced by Petit--Schapira~\cite{petit2020thickening}. 
See the body of the paper for a precise statement.

In \cref{section:SQ_diffeo_homeo}, we revisit sheaf quantization of Hamiltonian isotopies~\cite{GKS} and Hamiltonian stability~\cite{AI20}. 
Let $M$ be a $C^\infty$-manifold and $I$ be an open interval containing the closed interval $[0,1]$.
With the Hamiltonian isotopy $\phi^{H} \colon T^*M \times I \to T^*M$ associated with a timewise compactly supported function $H \colon T^*M \times I \to \bR$, one can associate a canonical object $K^{H} \in \SD(\bfk_{M^2 \times \bR_t \times I})$, called the sheaf quantization or the Guillermou--Kashiwara--Schapira (GKS) kernel. 
We prove that the restriction $K^H|_{M^2 \times \bR_t \times \{1\}}$ depends only on the time-1 map of $\phi^H_1$. 
This allows us to define the sheaf quantization $K^\varphi \in \SD(\bfk_{M^2 \times \bR_t})$ of a compactly supported Hamiltonian diffeomorphism $\varphi \in \Ham_c(T^*M)$. 
The sheaf quantization $K^{\varphi}$ defines an object $\cK^\varphi$ of the Tamarkin category $\cD(M^2)$. 
For these objects, we prove the following.

\begin{theorem}[{see \cref{theorem:stability_hofer}}]\label{theorem:stability_intro}
	For compactly supported Hamiltonian diffeomorphisms $\varphi, \varphi' \in \Ham_c(T^*M)$, one has 
	\begin{equation}
		d_{\cD(M^2)}(\cK^\varphi,\cK^{\varphi'}) \le d_H(\varphi,\varphi'),
	\end{equation}
	where the right-hand side denotes the Hofer distance between $\varphi$ and $\varphi'$.
\end{theorem}

\cref{theorem:stability_intro} is stronger than the previous Hamiltonian stability result in \cite{AI20} in the following two points: (1)~it is about the distance between GKS kernels and (2)~the distance $d_{\cD(X)}$ is possibly larger than that in \cite{AI20}. 

By the completeness result and the stability result, we can associate a sheaf with an element of the completion of $\Ham_c(T^*M)$ with respect to the Hofer metric. 
Indeed, a Cauchy sequence $(\varphi_n)_{n \in \bN} \subset \Ham_c(T^*M)$ with respect to the Hofer metric $d_H$ defines a Cauchy sequence $(\cK^{\varphi_n})_{n \in \bN} \subset \cD(M^2)$ with respect to $d_{\cD(M^2)}$ by \cref{theorem:stability_intro}.
Hence, by applying \cref{theorem:complete_intro}, we obtain a limit object $\cK^{[(\varphi_n)_{n}]} \in \cD(M^2)$. 
In particular, we can define the sheaf quantization $\cK^{\varphi_\infty}$ of a Hamiltonian homeomorphism $\varphi_\infty$ (see \cref{definition:hameo} for the definition).
One of the advantages of this approach is that we may directly use the limit object without explicitly dealing with a limit of some sequence to study $C^0$-symplectic geometry.

Next, in \cref{section:spectral}, we develop Lusternik--Schnirelmann theory in the microlocal theory of sheaves. 
More precisely, for an object $F \in \cD(M)$, we define the set of spectral invariants $\Spec(F) \subset \bR$ and prove the following. 
Here, $\cl(M)$ denotes the cup-length of $M$ over $\bfk$.

\begin{theorem}[{see \cref{theorem:sheaf_spectral_invariant} for a more precise statement}]\label{theorem:LS_intro}
	Let $t \colon M \times \bR_t \to \bR_t$ and $\pi_M \colon T^*(M \times \bR_t) \to M$ be the projections.
	Let $F \in \cD(M)$ and assume that $M$ is compact and $\RG_{M \times [-c,\infty)}(M \times \bR_t;F) \simeq \RG(M;\bfk_M)$ and $F|_{M \times (c,\infty)}$ is locally constant for $c \gg 0$.
	If $\# \Spec(F) \le \cl(M)$, then there exists $c \in \Spec(F)$ such that $\pi_M(\MS(F) \cap \Gamma_{dt} \cap \pi^{-1} t^{-1}(-c))$ is cohomologically non-trivial in $M$, where $\Gamma_{dt} \subset T^*(M \times \bR_t)$ is the graph of the $1$-form $dt$. 
	That is, for any open neighborhood $U$ of $\pi_M(\MS(F) \cap \Gamma_{dt} \cap \pi^{-1} t^{-1}(-c))$, the restriction map $\bigoplus_{n \ge 1} H^n(M;\bfk) \to \bigoplus_{n \ge 1} H^n(U;\bfk)$ is non-zero.
\end{theorem}

The theorem above was announced to appear in Humili\`ere--Vichery~\cite{HV}, which motivated our work.

Finally, in \cref{section:arnold}, by combining the machinery we have developed, we give a purely sheaf-theoretic proof of the following Arnold-type theorem for non-smooth objects (cf.\ Buhovsky--Humili\`ere--Seyfaddini~\cite{buhovsky2019arnold}).
We let $0_M$ denote the zero-section of $T^*M$.

\begin{theorem}[{see \cref{theorem:ineq_hameo}}]\label{theorem:inequality_intro}
	Assume that $M$ is compact and let $L$ be a compact exact Lagrangian submanifold of $T^*M$.
	Let $\varphi_\infty$ be a Hamiltonian homeomorphism of $T^*M$. 
	If the number of spectral invariants of $\varphi_\infty(L)$ is smaller than $\cl(M)+1$, then $0_M \cap \varphi_\infty(L)$ is cohomologically non-trivial in $M$, hence it is infinite.
\end{theorem}

This theorem is proved with the sheaf quantization of $L$ due to Guillermou~\cite{Gu12,Gu23} and Viterbo~\cite{Viterbo-Sheaves}, the sheaf quantization of $\varphi_\infty$, and \cref{theorem:LS_intro}.
By combining our machinery with the $C^0$-continuity of the spectral norm, which is obtained in the field of symplectic geometry, we can also construct a sheaf quantization of the image of the zero-section $0_M$ under a $C^0$-limit of Hamiltonian diffeomorphisms. 
Note that $C^0$-limits of Hamiltonian diffeomorphisms is a more general notion than Hamiltonian homeomorphisms. 
With the sheaf quantization, we can recover an Arnold-type theorem for a $C^0$-limit of Hamiltonian diffeomorphisms, which is exactly a result by Buhovsky--Humili\`ere--Seyfaddini~\cite{buhovsky2019arnold} (see \cref{proposition:arnold_hamhomeo}).

We also give a sheaf-theoretic proof of a Legendrian analogue of \cref{theorem:inequality_intro}. 
More precisely, we give an Arnold-type theorem for Hausdorff limits of Legendrian submanifolds in a 1-jet bundle (cf.\ \cite[Thm.~1.5]{buhovsky2019arnold}).

In \Cref{section:support_condition}, we give a Hamiltonian stability result with support conditions, which may be of independent interest.

\subsection{Related work}

The completeness of a persistence category was studied by Cruz~\cite{cruz2019metric} and Scoccola~\cite{scoccola2020locally}.
They showed that if the category admits any sequential (co)limit, then there is a limit object for any Cauchy sequence with respect to the interleaving distance. 
In this paper, we work with derived categories and need a different argument. 
Our completeness result (\cref{theorem:complete_intro}) also holds in a triangulated category endowed with a persistence structure.
See also Biran--Cornea--Zhang~\cite{biran2021triangulation} for persistence structures for triangulated categories.
Recently Fukaya~\cite{fukaya2021gromov} introduced the Gromov--Hausdorff distance between filtered $A_\infty$ categories, whose idea is based on the interleaving distance, and proved a completeness result.
During the preparation of this paper, the authors learned from St\'ephane Guillermou that he and Claude Viterbo had independently obtained a similar completeness result in a derived category of sheaves \cite{guillermou2022gamma}.

The persistence method has been widely applied to symplectic and contact geometry. 
After the pioneering work by Polterovich--Shelukhin~\cite{PS16}, persistence modules have been used to study barcodes of Floer cohomology complexes in Usher--Zhang~\cite{UZ16} and the study of spectral norms in Kislev--Shelukhin~\cite{kislev2018bounds}, to name a few. 
In this paper, we also investigate the relation between the interleaving-like distance $d_{\cD(M)}$ and the spectral norm (see \cref{proposition:spectral_norm_distance}).
Similar results are independently discovered by Guillermou and Viterbo~\cite{guillermou2022gamma}.

Kashiwara--Schapira~\cite{KS18persistent} studied persistence modules from the point of view of the microlocal theory of sheaves. 
Motivated by the work, Asano--Ike~\cite{AI20} introduced the interleaving-like distance $d_{\cD(M)}$ and showed that the distance between an object and its Hamiltonian deformation is upper bounded by the Hofer norm. 
The result was effectively used in Chiu~\cite{chiu2021quantum} and Li~\cite{li2021estimating}.
See also Zhang~\cite{Zhang20} for the interleaving-like distance in the Tamarkin category.

Buhovsky--Humili\`ere--Seyfaddini~\cite{buhovsky2018c} constructed a counterexample of Arnold's conjecture for a Hamiltonian homeomorphism of a closed symplectic manifold $M$ with $\dim M \ge 4$. 
However, one still obtains an Arnold-type theorem for a Hamiltonian homeomorphism if one reformulates the conjecture with the notion of spectral invariants, as in \cite{buhovsky2021action,kawamoto2019c,buhovsky2019arnold}.
In these studies, the $C^0$-continuity of persistence modules associated with Hamiltonian diffeomorphisms was effectively used (see also \cite{roux2018barcodes}). 
Guillermou~\cite{guillermou2013gromov} (see also \cite[Part~VII]{Gu23}) applied the microlocal theory of sheaves to $C^0$-symplectic geometry. 
He gave a purely sheaf-theoretic proof of the Gromov--Eliashberg theorem, using the involutivity of microsupports.

\subsection{Organization}

This paper is structured as follows. 
In \cref{section:microlocal_sheaf}, we recall some basics of the microlocal theory of sheaves. 
In \cref{section:tamarkin_SQ}, we first recall the interleaving distance for sheaves associated with a thickening kernel.
We then briefly review the Tamarkin category and sheaf quantization of Hamiltonian isotopies. 
In \cref{section:limit}, we prove the completeness of the derived category of sheaves with respect to the distance associated with a thickening kernel. 
This completeness in particular implies \cref{theorem:complete_intro}.
In \cref{section:SQ_diffeo_homeo}, we first prove a Hamiltonian stability theorem in terms of GKS kernels and the modified distance. 
We then show that the restriction of the sheaf quantization of a Hamiltonian isotopy to time $1$ depends only on the time-1 map, which allows us to define the sheaf quantization of a Hamiltonian diffeomorphism.
We also prove \cref{theorem:stability_intro} and construct a sheaf quantization of a Hamiltonian homeomorphism.
In \cref{section:spectral}, we develop Lusternik--Schnirelmann theory for the Tamarkin category and prove \cref{theorem:LS_intro}.
We prove \cref{theorem:inequality_intro} in \cref{section:arnold} and its Legendrian analogue in \cref{section:legendrian}.
In \Cref{section:support_condition}, we prove a Hamiltonian stability result with support conditions, by using sheaf quantization of 2-parameter Hamiltonian isotopies.

\subsection*{Acknowledgments}

The authors would like to express their gratitude to Vincent Humili\`ere for fruitful discussion on many parts of this paper. 
They also thank St\'ephane Guillermou and Takahiro Saito for helpful discussion, and Yusuke Kawamoto for drawing their attention to applications of the microlocal theory of sheaves to $C^0$-symplectic geometry. 
They are also grateful to Wenyuan Li, Tatsuki Kuwagaki, Morimichi Kawasaki, and Pierre Schapira for their helpful comments.
The authors also thank the anonymous referees for their helpful comments, which improved many parts of the paper.
TA was supported by Innovative Areas Discrete Geometric Analysis for Materials Design (Grant No.~17H06461).
YI was supported by JSPS KAKENHI (Grant No.~21K13801 and 22H05107) and ACT-X, Japan Science and Technology Agency (Grant No.~JPMJAX1903).

\section{Microlocal theory of sheaves}\label{section:microlocal_sheaf}

Throughout this paper, let $\bfk$ be a field.
We mainly follow the notation of \cite{KS90}. 
In this section, let $X$ be a $C^\infty$-manifold without boundary.

\subsection{Geometric notions}\label{subsection:geometric}

For a locally closed subset $Z$ of $X$, we denote by $\overline{Z}$
its closure and by $\Int(Z)$ its interior.
We also denote by $\delta_X \colon X \to X \times X, x \mapsto (x,x)$ the diagonal map and by $\Delta_X \coloneqq \delta_X(X)$ the diagonal of $X \times X$.
We often write $\Delta$ for $\Delta_X$ for simplicity. 
We denote by $TX$ the tangent bundle and by $T^*X$ the cotangent bundle of $X$.
We write $\pi \colon T^*X \to X$ for the projection.
For a closed submanifold $M$ of $X$, we denote by $T^*_MX$ the conormal bundle to $M$ in $X$.
In particular, $T^*_XX$ denotes the zero-section of $T^*X$, which we often simply write $0_X$.
We set $\rT X\coloneqq T^*X \setminus 0_X$.

With a morphism of manifolds $f \colon X \to Y$, we associate the following commutative diagram of morphisms of manifolds:
\begin{equation}\label{diag:fpifd}
	\begin{aligned}
		\xymatrix{
			T^*X \ar[d]_-{\pi} & X \times_Y T^*Y \ar[d] \ar[l]_-{f_d}
			\ar[r]^-{f_\pi} & T^*Y \ar[d]^-{\pi} \\
			X \ar@{=}[r] & X \ar[r]_-f & Y,
		}
	\end{aligned}
\end{equation}
where $f_\pi$ is the projection and $f_d$ is induced by the transpose
of the tangent map $f' \colon TX \to X \times_Y TY$.

We denote by $(x;\xi)$ a local homogeneous coordinate system on
$T^*X$.
The cotangent bundle $T^*X$ is an exact symplectic manifold with the
Liouville 1-form $\alpha_{T^*X}=\langle \xi, dx \rangle$.
Thus the symplectic form on $T^*X$ is defined to be $\omega=d\alpha_{T^*X}$.
We denote by $a \colon T^*X \to T^*X,(x;\xi) \mapsto (x;-\xi)$ the
antipodal map.
For a subset $A$ of $T^*X$, $A^a$ denotes its image under the antipodal map $a$.
A subset $A$ of $T^*M$ is said to be \emph{conic} if it is invariant under the action of $\bR_{>0}$ on $T^*M$, that is, the scaling of the fibers.

\subsection{Microsupports of sheaves}

We write $\bfk_X$ for the constant sheaf with stalk $\bfk$ and $\Module(\bfk_X)$ for the abelian category of sheaves of $\bfk$-vector spaces on $X$.
Moreover, we denote by $\SD(\bfk_X)$ the unbounded derived category of $\bfk$-vector spaces on $X$. 
Although all the results are stated for bounded derived categories in \cite{KS90}, we can apply most of them for unbounded categories, which we shall state in this subsection. 
We refer to \cite{MR932640,KS06,robalo2018lemma}.
One can define Grothendieck's six operations $\cRHom,\allowbreak
\otimes, \allowbreak \rR f_*,\allowbreak f^{-1},\allowbreak
\rR f_!,\allowbreak f^!$ for a continuous map $f \colon X \to Y$ with suitable conditions.
For a locally closed subset $Z$ of $X$, we denote by $\bfk_Z$ the
zero extension of the constant sheaf with stalk $\bfk$ on $Z$ to $X$,
whose stalk is $0$ on $X \setminus Z$.
Moreover, for a locally closed subset $Z$ of $X$ and $F \in \SD(\bfk_X)$,
we define $F_Z, \RG_Z(F) \in \SD(\bfk_X)$ by
\begin{equation}
	F_Z\coloneqq F \otimes \bfk_Z, \quad \RG_Z(F)\coloneqq \cRHom(\bfk_Z,F).
\end{equation}

Let us recall the definition of the \emph{microsupport} $\MS(F)$ of an
object $F \in \SD(\bfk_X)$. 

\begin{definition}[{\cite[Def.~5.1.2]{KS90}}]\label{definition:microsupport}
	Let $F \in \SD(\bfk_X)$ and $p \in T^*X$.
	One says that $p \not\in \MS(F)$ if there is a neighborhood $U$ of
	$p$ in $T^*X$ such that for any $x_0 \in X$ and any
	$C^\infty$-function $\varphi$ on $X$ (defined on a neighborhood of
	$x_0$) satisfying $d\varphi(x_0) \in U$, one has
	$\RG_{\{ x \in X \mid \varphi(x) \ge \varphi(x_0)\}}(F)_{x_0} \simeq 0$.
	One also sets $\rMS(F) \coloneqq \MS(F) \cap \rT X$.
\end{definition}

For a closed subset $A$ of $T^*X$, we denote by $\SD_{A}(\bfk_X)$ the full triangulated subcategory of $\SD(\bfk_X)$ consisting of objects whose microsupports are contained in $A$.

The following is called the microlocal Morse lemma.

\begin{proposition}[{\cite[Prop.~5.4.17]{KS90} and \cite[Thm.~4.1]{robalo2018lemma}}]\label{proposition:microlocalmorse}
	Let $F \in \SD(\bfk_X)$ and $\varphi \colon X \to \bR$ be a
	$C^\infty$-function.
	Let moreover $a,b \in \bR$ with $a<b$.
	Assume
	\begin{enumerate}
		\renewcommand{\labelenumi}{$\mathrm{(\arabic{enumi})}$}
		\item $\varphi$ is proper on $\Supp(F)$,
		\item  $d\varphi(x) \not\in \MS(F)$ for any $x \in
		\varphi^{-1}([a,b))$.
	\end{enumerate}
	Then the canonical morphism
	\begin{equation}
		\RG(\varphi^{-1}((-\infty,b));F)
		\to
		\RG(\varphi^{-1}((-\infty,a));F)
	\end{equation}
	is an isomorphism.
\end{proposition}

We consider the behavior of the microsupports with respect to functorial operations.

\begin{proposition}[{\cite[Prop.~5.4.4, 
		Prop.~5.4.13, and Prop.~5.4.5]{KS90}}]\label{proposition:SSpushpull}
	Let $f \colon X \to Y$ be a morphism of manifolds, $F \in \SD(\bfk_X)$,
	and $G \in \SD(\bfk_Y)$.
	\begin{enumerate}
		\item Assume that $f$ is proper on $\Supp(F)$.
		Then $\MS(\rR f_*F) \subset f_\pi f_d^{-1}(\MS(F))$.
		Moreover, if $f$ is a closed embedding, the inclusion is an equality.
		\item Assume that $f$ is non-characteristic for $\MS(G)$ (see \cite[Def.~5.4.12]{KS90} for the definition).
		Then $\MS(f^{-1}G) \cup \MS(f^!G) \subset
		f_d f_\pi^{-1}(\MS(G))$.
	\end{enumerate}
\end{proposition}

For closed conic subsets $A$ and $B$ of $T^*X$, let us denote by $A+B$
the fiberwise sum of $A$ and $B$, that is,
\begin{equation}
	A+B
	\coloneqq 
	\left\{
	(x;a+b) \; \middle| \; 
	\begin{aligned}
		& x \in \pi(A) \cap \pi(B), \\
		& a \in A \cap \pi^{-1}(x), b \in B \cap \pi^{-1}(x) 
	\end{aligned}
	\right\} 
	\subset T^*X.
\end{equation}

\begin{proposition}[{\cite[Prop.~5.4.14]{KS90}}]\label{proposition:SStenshom}
	Let $F, G \in \SD(\bfk_X)$.
	\begin{enumerate}
		\item If $\MS(F) \cap \MS(G)^a \subset 0_X$,
		then $\MS(F \otimes G) \subset \MS(F)+\MS(G)$.
		\item If $\MS(F) \cap \MS(G) \subset 0_X$,
		then $\MS(\cRHom(F,G)) \subset \MS(F)^a+\MS(G)$.
	\end{enumerate}
\end{proposition}

We need an estimate for the microsupport of a kind of limit object in $\SD(\bfk_X)$.
For that purpose, we use the following estimates.

\begin{lemma}[{cf.\ \cite[Exe.~V.7]{KS90}}]\label{lem:ms-prod-sum}
	Let $I$ be an index set and $F_i \in \SD(\bfk_X)$ for $i \in I$. 
	Then, one has 
	\begin{align}
		\MS\left(\prod_{i \in I} F_i\right), \ \MS \left(\bigoplus_{i \in I} F_i \right) \subset \overline{\bigcup_{i \in I} \MS(F_i)}. 
	\end{align}
\end{lemma}

\subsection{Composition and convolution}\label{subsection:composition}

We recall the operation called the composition of sheaves.

For $i=1,2,3$, let $X_i$ be a manifold.
We write $X_{ij}\coloneqq X_i \times X_j$ and $X_{123}\coloneqq X_1 \times X_2 \times X_3$ for short.
We denote by $q_{ij}$ the projection $X_{123} \to X_{ij}$.
Similarly, we denote by $p_{ij}$ the projection $T^*X_{123} \to T^*X_{ij}$.
We also denote by $p_{12^a}$ the composite of $p_{12}$ and the antipodal map on $T^*X_2$.

Let $A \subset T^*X_{12}$ and $B \subset T^*X_{23}$.
We set
\begin{equation}\label{equation:compset}
	A \circ B
	\coloneqq 
	p_{13}(p_{12^a}^{-1}A \cap p_{23}^{-1}B) \subset T^*X_{13}.
\end{equation}
We define a composition operation of sheaves by
\begin{equation}
	\begin{aligned}
		\underset{X_2}{\circ} \colon \SD(\bfk_{X_{12}}) \times \SD(\bfk_{X_{23}}) 
		& \to \SD(\bfk_{X_{13}}), \\
		(K_{12},K_{23}) 
		& \mapsto K_{12} \underset{X_2}{\circ} K_{23}
		\coloneqq 
		\rR{q_{13}}_! (q_{12}^{-1}K_{12}\otimes q_{23}^{-1}K_{23}).
	\end{aligned}
\end{equation}
If there is no risk of confusion, we simply write $\circ$ instead of
$\underset{X_2}{\circ}$.
By \cref{proposition:SSpushpull,proposition:SStenshom}, we have the following.

\begin{proposition}\label{proposition:SScomp}
	Let $K_{ij} \in \SD(\bfk_{X_{ij}})$ and set $\Lambda_{ij}\coloneqq \MS(K_{ij}) \subset T^*X_{ij} \ (ij=12,23)$.
	Assume
	\begin{enumerate}
		\renewcommand{\labelenumi}{$\mathrm{(\arabic{enumi})}$}
		\item $q_{13}$ is proper on $q_{12}^{-1}\Supp(K_{12}) \cap q_{23}^{-1}\Supp(K_{23})$,
		\item $p_{12^a}^{-1}\Lambda_{12} \cap p_{23}^{-1}\Lambda_{23} \cap (T^*_{X_1}X_1 \times T^*X_2 \times T^*_{X_3}X_3) \subset 0_{X_{123}}$.
	\end{enumerate}
	Then
	\begin{equation}
		\MS(K_{12} \underset{X_2}{\circ} K_{23}) \subset
		\Lambda_{12} \circ \Lambda_{23}.
	\end{equation}
\end{proposition}

In this work, we often use sheaves on $X \times \bR$.
We introduce the operation of convolution for objects of $\SD(\bfk_{X \times \bR})$.
Define the maps
\begin{gather}\label{eq:symbol_comv}
	\tilde{q}_1,\tilde{q}_2, m \colon X \times \bR \times \bR \to X \times \bR, \\
	\notag \tilde{q}_1(x,t_1,t_2)=(x,t_1), \ \tilde{q}_2(x,t_1,t_2)=(x,t_2), \ m(x,t_1,t_2)=(x,t_1+t_2).
\end{gather}
We define a convolution operation by 
\begin{equation}
	\begin{aligned}
		\star \colon \SD(\bfk_{X \times \bR}) \times \SD(\bfk_{X \times \bR}) 
		& \to \SD(\bfk_{X \times \bR}), \\
		(F,G) 
		& \mapsto F \star G
		\coloneqq 
		\rR m_! (\tilde{q}_1^{-1}F \otimes \tilde{q}_2^{-1}G).
	\end{aligned}
\end{equation}
We also introduce a right adjoint to the convolution functor.
Set $i \colon X \times \bR \to X \times \bR, (x,t) \mapsto (x,-t)$.

\begin{definition}
	For $F,G \in \SD(\bfk_{X \times \bR})$, one sets
	\begin{align}
		\cHom^\star(F,G) & \coloneqq \rR \tilde{q}_{1*} \cRHom(\tilde{q}_2^{-1}F,m^!G) \\
		& \ \simeq \rR m_*\cRHom(\tilde{q}_2^{-1}i^{-1}F,\tilde{q}_1^!G).
	\end{align}
\end{definition}

\begin{lemma}\label{lem:conv-hom-MS}
	Let $F,G \in \SD(\bfk_{X \times \bR})$ and assume that there exist two closed cones $A,B \subset \bR$ such that $\MS(F) \subset T^*X \times \bR \times A$ and $\MS(G) \subset T^*X \times \bR \times B$.
	Then, one has 
	\begin{equation}
		\begin{aligned}
			\MS(F \star G) \subset T^*X \times \bR \times (A \cap B), \\
			\MS(\cHom^\star(F,G)) \subset T^*X \times \bR \times (A \cap B).
		\end{aligned}
	\end{equation}
\end{lemma}

It is useful to define a more general operation, which combines composition and convolution. 
Set 
\begin{align}
	& \tilde{q}_{12} \colon X_{123} \times \bR^2 \to X_{12} \times \bR, 
	(x_1,x_2,x_3,t_1,t_2) \mapsto (x_1,x_2,t_1), \\
	& \tilde{q}_{23} \colon X_{123} \times \bR^2 \to X_{23} \times \bR, 
	(x_1,x_2,x_3,t_1,t_2) \mapsto (x_2,x_3,t_2), \\
	& m_{13} \colon X_{123} \times \bR^2 \to X_{13} \times \bR, 
	(x_1,x_2,x_3,t_1,t_2) \mapsto (x_1,x_3,t_1+t_2)
\end{align}
and define
\begin{equation}
	\begin{aligned}
		\bullet_{X_2} \colon \SD(\bfk_{X_{12} \times \bR}) \times \SD(\bfk_{X_{23} \times \bR}) 
		& \to \SD(\bfk_{X_{13} \times \bR}), \\
		(K_{12},K_{23}) 
		& \mapsto K_{12} \underset{X_2}{\bullet} K_{23}
		\coloneqq 
		\rR{m_{13}}_!\,(\tilde{q}_{12}^{-1}K_{12}\otimes \tilde{q}_{23}^{-1}K_{23}).
	\end{aligned}
\end{equation}
If there is no risk of confusion, we simply write $\bullet$ instead of
$\underset{X_2}{\bullet}$.

\subsection{Homotopy colimits}

Here, we recall the definition of homotopy colimits (cf.\ \cite{bokstedt1993homotopy} and \cite{KS06}) and estimate their microsupports.

Let $(F_n)_{n \in \bN}$ be a sequence of objects of $\SD(\bfk_X)$ together with morphisms $f_n \colon F_n \to F_{n+1} \ (n \in \bN=\bZ_{\ge 0})$, that is, $(F_n, f_n)_{n \in \bN}$ is an inductive system in $\SD(\bfk_X)$. 
Define a morphism $s \colon \bigoplus_{n \in \bN}F_n \to \bigoplus_{n \in \bN}F_n$ as the composite 
\begin{equation}
	\bigoplus_{n \in \bN} F_n \xrightarrow{\bigoplus_n f_n} \bigoplus_{n \in \bN} F_{n+1} \simeq \bigoplus_{n \in \bZ_{\ge 1}} F_n \to \bigoplus_{n \in \bN} F_n.
\end{equation}
Then one can define the homotopy colimit of the inductive system $(F_n, f_n)_{n \in \bN}$ as the cone of the morphism 
\begin{equation}
	\id-s \colon \bigoplus_{n \in \bN} F_n \to \bigoplus_{n \in \bN} F_n,
\end{equation}
which we write $\hocolim(F_n) \in \SD(\bfk_X)$.
We have a canonical morphism $\rho_n \colon F_n \to \hocolim(F_n)$. 
Given a sequence of morphisms $g_n \colon F_n \to G \ (n \in \bN)$ such that $g_{n+1} \circ f_n = g_n$ for any $n$, we get a morphism $g \colon \hocolim(F_n) \to G$ satisfying $g_n = g \circ \rho_n$.

\begin{lemma}\label{lemma:ms_hocolim}
	Let $(F_n, f_n)_{n \in \bN}$ be an inductive system in $\SD(\bfk_X)$. 
	Then 
	\begin{equation}
		\MS(\hocolim(F_n)) \subset \bigcap_{N \in \bN} \overline{\bigcup_{n \ge N} \MS(F_n)}.
	\end{equation}
\end{lemma}

\begin{proof}
	By \cref{lem:ms-prod-sum}, $\MS(\bigoplus_{n \in \bN} F_n) \subset \overline{\bigcup_n \MS(F_n)}$. 
	Note that we have $\hocolim_n(F_n) \simeq \hocolim_n(G_n)$ with $G_n=F_{n+N}$ for any $N \in \bN$.     
	Hence, the result follows from the triangle inequality for microsupports and the definition of homotopy colimits.
\end{proof}

\section{Interleaving distance for sheaves and sheaf quantization}\label{section:tamarkin_SQ}

In this section, we review the interleaving distance for sheaves following Petit--Schapira~\cite{petit2020thickening}.
We also briefly review the Tamarkin category \cite{Tamarkin} and sheaf quantization of Hamiltonian isotopies \cite{GKS}.
See also Guillermou--Schapira~\cite{GS14} and Zhang~\cite{Zhang20} for details of the Tamarkin category.

\subsection{Interleaving distance for sheaves}\label{subsection:thickening}

We recall some notions from Petit--Schapira~\cite{petit2020thickening}.
A topological space is said to be \emph{good} if it is Hausdorff, locally compact, countable at infinity, and finite flabby dimension.

\begin{definition}
	Let $X$ be a good topological space. 
	A \emph{thickening kernel} on $X$ is a monoidal presheaf $\frakK$ on $(\bR_{\ge 0}, +)$ with values in the monoidal category $\SD(\bfk_{X \times X})$. 
	In other words, it is a family of kernels $\frakK_a \in \SD(\bfk_{X \times X})$ with a morphism $\rho_{b,a} \colon \frakK_b \to \frakK_a$ for $a \le b$ together with isomorphisms
	\begin{equation}
		\frakK_a \circ \frakK_b \simeq \frakK_{a+b}, \quad \frakK_0 \simeq \bfk_{\Delta_X}
	\end{equation}
	satisfying the compatibility conditions (see \cite[Def.~1.2.2]{petit2020thickening} for details).
\end{definition}

For a thickening kernel $\frakK$ on $X$ and $F \in \SD(\bfk_X)$, we write $\rho_{b,a}(F)$ for the morphism $\rho_{b,a} \circ \id_F \colon \frakK_b \circ F \to \frakK_a \circ F \ (a \le b)$.

\begin{definition}
	Let $\frakK$ be a thickening kernel on $X$, $F,G \in \SD(\bfk_X)$, and $a,b \in \bR_{\ge 0}$.
	\begin{enumerate}
		\item The pair $(F,G)$ is said to be \emph{$(a,b)$-isomorphic} if there exist morphisms $\alpha \colon \frakK_a \circ F \to G$ and $\beta \colon \frakK_b \circ G \to F$ such that 
		\begin{enumerate}
			\renewcommand{\labelenumii}{$\mathrm{(\arabic{enumii})}$}
			\item the composite $\frakK_{a+b} \circ F \xrightarrow{\frakK_a \circ \alpha} \frakK_b \circ G \xrightarrow{\beta} F$ is equal to $\rho_{a+b,0}(F)$,
			\item the composite $\frakK_{a+b} \circ G \xrightarrow{\frakK_b \circ \beta} \frakK_a \circ F \xrightarrow{\alpha} G$ is equal to $\rho_{a+b,0}(G)$.
		\end{enumerate}
		In this case, the pair of morphisms $(\alpha,\beta)$ is called an $(a,b)$-isomorphism. 
		
		If $(F,G)$ is $(a,a)$-isomorphic then $F$ and $G$ are called $a$-isomorphic.
		\item One sets 
		\begin{equation}
			d_\frakK(F,G)
			\coloneqq 
			\inf 
			\left\{	a+b \in \bR_{\ge 0} 
			\; \middle| \;
			\begin{aligned}
				& a,b \in \bR_{\ge 0}, \\
				& \text{$(F,G)$ is $(a,b)$-isomorphic}
			\end{aligned}
			\right\},
		\end{equation}
		which defines a pseudo-distance on the category $\SD(\bfk_X)$.
		\item The pair $(F,G)$ is said to be \emph{weakly $(a,b)$-isomorphic} if there exist morphisms $\alpha, \delta \colon \frakK_a \circ F \to G$ and $\beta, \gamma \colon \frakK_b \circ G \to F$ such that 
		\begin{enumerate}
			\renewcommand{\labelenumii}{$\mathrm{(\arabic{enumii})}$}
			\item the composite $\frakK_{a+b} \circ F \xrightarrow{\frakK_a \circ \alpha} \frakK_b \circ G \xrightarrow{\beta} F$ is equal to $\rho_{a+b,0}(F)$,
			\item the composite $\frakK_{a+b} \circ G \xrightarrow{\frakK_b \circ \gamma} \frakK_a \circ F \xrightarrow{\delta} G$ is equal to $\rho_{a+b,0}(G)$, 
			\item $\alpha \circ \rho_{2a,a}(F) = \delta \circ \rho_{2a,a}(F)$ and $\beta \circ \rho_{2b,b}(G)= \gamma \circ \rho_{2b,b}(G)$.
		\end{enumerate}
		\item One says that $F$ is \emph{$a$-torsion} or \emph{$a$-trivial} if $\rho_{a,0}(F) \colon \frakK_a \circ F \to F$ is the zero morphism.
	\end{enumerate}
\end{definition}

\begin{remark}\label{remark:distances}
	\begin{enumerate}
		\item One can see that
		\begin{equation}
			\text{$(a,b)$-isomorphic} \Rightarrow \text{weakly $(a,b)$-isomorphic}
		\end{equation}
		and 
		\begin{equation}
			\begin{aligned}
				\text{weakly $(a,b)$-isomorphic} & \Rightarrow \text{$(2a,2b)$-isomorphic} \\
				& \Rightarrow \text{$2\max(a,b)$-isomorphic}.
			\end{aligned}
		\end{equation}
		\item In \cite{petit2020thickening}, the authors define 
		\begin{equation}
			\dist_{\frakK}(F,G) \coloneqq \inf\{ a \in \bR_{\ge 0} \mid \text{$F$ and $G$ are $a$-isomorphic} \}
		\end{equation}
		and call it the \emph{interleaving distance} associated with $\frakK$.
		By (i), the pseudo-distances $d_\frakK$ and $\dist_\frakK$ are equivalent.
		Indeed, 
		\begin{equation}
			d_\frakK(F,G) \le 2\dist_\frakK(F,G) \le 2d_\frakK(F,G).
		\end{equation}
		\item The interleaving distance above is a generalization of the convolution distance $d_C$ on $\SD(\bfk_\bR)$ introduced by Kashiwara--Schapira~\cite{KS18persistent} and later investigated by \cite{kashiwara2021piecewise,berkouk2021derived,berkouk2021ephemeral,berkouk2019level} and others.
		Indeed, when $X=\bR$ and $\frakK_a = \bfk_{\Delta_a}$, where $\Delta_a \coloneqq \{ (x,y) \in \bR^2 \mid \|x-y\| \le a \} \subset \bR^2$ is the thickened diagonal, we find that $d_C=\dist_\frakK$.
	\end{enumerate}
\end{remark}

In what follows, let $\frakK$ be a thickening kernel on $X$. 
The statement of the following lemma is slightly stronger than \cite[Lem.~4.14]{AI20}, but the proof itself is almost the same.
We need the stronger result in this paper, so we reproduce the proof for the convenience of the reader.

\begin{lemma}[{cf.\ \cite[Lem.~4.14]{AI20}}]\label{lemma:torsion_to_weakisom}
	Let $F \stackrel{u}{\lto} G \stackrel{v}{\lto} H \stackrel{w}{\lto} F[1]$ be an exact triangle in $\SD(\bfk_X)$ and assume that $F$ is $c$-torsion.
	Then $(G,H)$ is weakly $(0,c)$-isomorphic.
\end{lemma}

\begin{proof}
	By assumption, we have $w \circ \rho_{c,0}(H)= \rho_{c,0}(F[1]) \circ (\frakK_c \circ_X w)=0$.
	Hence, we get a morphism $\gamma \colon \frakK_c \circ H \to G$ satisfying $\rho_{c,0}(H)=v \circ \gamma$:
	\begin{equation}
		\begin{aligned}
			\xymatrix{
				\frakK_c \circ F \ar[r]^-{\frakK_c \circ u} \ar[d] & \frakK_c \circ G \ar[r]^-{\frakK_c \circ v} \ar[d] \ar@{}[rd]|(.7){\circlearrowright} & \frakK_c \circ H \ar[r]^-{\frakK_c \circ w} \ar[d] \ar[d] \ar@{-->}[ld]_-{\gamma} & \frakK_c \circ F[1] \ar[d]^-0 \\
				F \ar[r]_-{u} & G \ar[r]_-{v} & H \ar[r]_-{w} & F[1].
			}
		\end{aligned}
	\end{equation}
	On the other hand, since $\rho_{c,0}(G) \circ (\frakK_c \circ_X u)=u \circ \rho_{c,0}(F)=0$, there exists a morphism $\beta \colon \frakK_c \circ H \to G$ satisfying $\rho_{c,0}(G)=\beta \circ (\frakK_c \circ_X v)$:
	\begin{equation}
		\begin{aligned}
			\xymatrix{
				\frakK_c \circ F \ar[r]^-{\frakK_c \circ u} \ar[d]_-0 & \frakK_c \circ G \ar[r]^-{\frakK_c \circ v} \ar[d] \ar@{}[rd]|(.3){\circlearrowright} & \frakK_c \circ H \ar[r]^-{\frakK_c \circ w} \ar[d] \ar[d] \ar@{-->}[ld]^-{\beta} & \frakK_c \circ F[1] \ar[d] \\
				F \ar[r]_-{u} & G \ar[r]_-{v} & H \ar[r]_-{w} & F[1].
			}
		\end{aligned}
	\end{equation}
	Moreover, we obtain 	
	\begin{equation}
		\begin{aligned}
			\beta \circ \rho_{2c,c}(H)
			& = 
			\beta \circ (\frakK_c \circ_X \rho_{c,0}(H)) \\
			& = 
			\beta \circ (\frakK_c \circ_X v) \circ (\frakK_c \circ_X \gamma) \\
			& =
			\rho_{c,0}(G) \circ (\frakK_c \circ_X \gamma) \\
			& =
			\gamma \circ (\frakK_c \circ_X \rho_{c,0}(H)) 
			=
			\gamma \circ \rho_{2c,c}(H),
		\end{aligned}		
	\end{equation}
	which proves the lemma.
\end{proof}

In particular, if $F \to G \to H \toone$ is an exact triangle, $F$ is $a$-torsion, and $G$ is $b$-torsion, then $H$ is $(a+b)$-torsion.

In our later applications, we mainly focus on the interleaving distance on the derived category $\SD(\bfk_{X \times \bR_t})$.
For $c \in \bR_{\ge 0}$, define 
\begin{equation}
	\frakK_c \coloneqq \bfk_{\Delta_X \times \Delta_c} \in \SD(\bfk_{(X \times \bR)^2}),
\end{equation}
where $\Delta_c \coloneqq \{ (t_1,t_2) \mid \|t_1-t_2\| \le c \} \subset \bR^2$ is the thickened diagonal of $\Delta_\bR \subset \bR^2$.
Then the assignment $c \mapsto \frakK_c$ defines a thickening kernel on $X \times \bR_t$.
Hence, we can consider the pseudo-distance $d_\frakK$ on $\SD(\bfk_{X \times \bR_t})$ associated with $\frakK$, which we denote by $d_{X \times \bR_t}$.
This is a slight modification of the relative distance $\dist_{X \times \bR_t/X}$ studied in \cite{petit2020thickening}.
Note also that $\frakK_c \circ F \simeq \bfk_{X \times [-c,c]} \star F$. 
With the notation in \cref{subsection:composition}, for $F,F'\in \SD(\bfk_{X_{12} \times \bR_t})$ and $G,G'\in \SD(\bfk_{X_{23} \times \bR_t})$, 
\begin{equation}\label{eqn:bullet_distance}
	d_{X_{13} \times \bR_t}(F \bullet G,F' \bullet G') \le 
	d_{X_{12} \times \bR_t}(F, F')+d_{X_{23} \times \bR_t}(G,G').
\end{equation}

\subsection{Tamarkin category}\label{subsec:tamarkin}

Let $X$ be a manifold without boundary.
We let $(t;\tau)$ denote the homogeneous coordinate system on $T^*\bR_t$. 
It is proved by Tamarkin~\cite{Tamarkin} that the functor $P_l \coloneqq \bfk_{X \times [0,\infty)} \star (\ast) \colon \SD(\bfk_{X \times \bR_t}) \to \SD(\bfk_{X \times \bR_t})$ defines a projector onto ${}^{\bot} \SD_{\{\tau \le 0\}}(\bfk_{X \times \bR_t})$, where $\{ \tau \le 0 \}= \{(x,t;\xi,\tau) \mid \tau \le 0 \} \subset T^*(X \times \bR_t)$ and ${}^\bot (\ast)$ denotes the left orthogonal.

\begin{definition}
	One defines
	\begin{equation}
		\cD(X) \coloneqq {}^{\bot} \SD_{\{\tau \le 0\}}(\bfk_{X \times \bR_t}),
	\end{equation}
	and call it the \emph{Tamarkin category} of $X$. 
\end{definition}

For an object $F \in \SD(\bfk_{X \times \bR_t})$, $F \in \cD(X)$ if and only if $P_l(F) \simeq F$. 
Note also that $\cD(X) \subset \SD_{\{\tau \ge 0 \}}(\bfk_{X \times \bR_t})$ by \cref{lem:conv-hom-MS}.

\begin{definition}\label{definition:distance_tamarkin}
	One defines $d_{\cD(X)}$ as the restriction of the pseudo-distance $d_{X \times \bR_t}$ on $\SD(\bfk_{X \times \bR_t})$ to $\cD(X)$. 
\end{definition}

We will describe the pseudo-distance using the translation to the $\bR_t$-direction. 
For $c \in \bR$, let $T_c \colon X \times \bR_t \to X \times \bR_t, (x,t) \mapsto (x,t+c)$ be the translation map by $c$ to the $\bR$-direction.
In what follows, we write $T_c$ instead of ${T_c}_*$ for simplicity.
Recall that we have set $\frakK_c \coloneqq \bfk_{\Delta_X \times \Delta_c} \in \SD(\bfk_{(X \times \bR)^2})$.

\begin{lemma}\label{lem:shift}
	Let $F \in \SD_{\{ \tau \ge 0 \}}(\bfk_{X \times \bR_t})$.
	Then $\frakK_c \circ F \simeq {T_{-c}} F$.
\end{lemma}

\begin{proof}
	First we recall that $\frakK_c \circ F \simeq \bfk_{X \times [-c,c]} \star F$. 
	Since there exist an exact triangle $\bfk_{X \times (-c,c]} \to \bfk_{X \times [-c,c]} \to \bfk_{X \times \{-c\}} \toone$ and an isomorphism ${T_{-c}} F \simeq \bfk_{X \times \{-c\}} \star F$, it suffices to show that $\bfk_{X \times (-c,c]} \star F \simeq 0$.
	By \cref{lem:conv-hom-MS-const}, there exists $H \in \SD(\bfk_X)$ such that $\bfk_{X \times (-\infty,0]} \star F \simeq H \boxtimes \bfk_{\bR_t}$.
	Thus, we obtain 
	\begin{equation}
		\begin{aligned}
			\bfk_{X \times (-c,c]} \star F 
			& \simeq 
			\bfk_{X \times (-c,c]} \star \bfk_{X \times (-\infty,0]} \star F \\
			& \simeq 
			\bfk_{X \times (-c,c]} \star (H \boxtimes \bfk_{\bR_t}) \\
			& \simeq H \boxtimes (\bfk_{(-c,c]} \star \bfk_{\bR_t}) 
			\simeq 0,
		\end{aligned}
	\end{equation}
	which completes the proof.
\end{proof}

Let $F \in \SD_{\{ \tau \ge 0\}}(\bfk_{X \times \bR_t})$.
Then, we have an isomorphism
\begin{equation}
	\rR m_*(\tilde{q}_1^{-1} \bfk_{X \times [0,\infty)}\otimes \tilde{q}_2^{-1}F) \simto 
	\rR m_*(\tilde{q}_1^{-1} \bfk_{X \times \{0\}}\otimes \tilde{q}_2^{-1}F)
	\simeq 
	F.
\end{equation}
Hence, for $c, d \in \bR_{\ge 0}$ with $c \le d$, the canonical morphism $\bfk_{X \times [c,\infty)} \to \bfk_{X \times [d,\infty)}$ induces a canonical morphism
\begin{equation}
	\begin{aligned}
		\tau_{c,d}(F) \colon & {T_c} F \simeq 
		\rR m_*(\tilde{q}_1^{-1} \bfk_{X \times [c,\infty)}\otimes \tilde{q}_2^{-1}F) \\
		& \to 
		\rR m_*(\tilde{q}_1^{-1} \bfk_{X \times [d,\infty)}\otimes \tilde{q}_2^{-1}F)
		\simeq 
		{T_d} F.
	\end{aligned}
\end{equation}
By \cref{lem:shift}, the morphism is identified with 
\begin{equation}
	T_{c+d}\rho_{d,c}(F) \colon T_{c+d}(\frakK_d \circ F) \to T_{c+d}(\frakK_c \circ F).
\end{equation}
Hence, a pair $(F,G)$ of objects of $\cD(X)$ is $(a,b)$-isomorphic if and only if there exist morphisms $\alpha, \delta \colon F \to {T_a} G$ and $\beta, \gamma \colon G \to {T_b} F$ such that
\begin{enumerate}
	\renewcommand{\labelenumi}{$\mathrm{(\arabic{enumi})}$}
	\item $F \xrightarrow{\alpha} {T_a} G \xrightarrow{{T_a} \beta} {T_{a+b}} F$ is equal to $\tau_{0,a+b}(F) \colon F \to {T_{a+b}} F$ and
	\item $G \xrightarrow{\beta} {T_b} F \xrightarrow{{T_b} \alpha} {T_{a+b}} G$ is equal to $\tau_{0,a+b}(G) \colon G \to {T_{a+b}} G$.
\end{enumerate}
In this form, we can see that $d_{\cD(X)}$ is similar to the pseudo-distance introduced in \cite{AI20} (see \cref{remark:ab_relation} below).

\begin{remark}\label{remark:ab_relation}
	The terminology has been changed from that in \cite{AI20}.
	In that paper, ``weakly $(a,b)$-isomorphic" in this paper was called ``$(a,b)$-isomorphic". 
	Moreover, we defined the notion of ``$(a,b)$-interleaved" as follows:
	a pair $(F,G)$ of objects of $\cD(X)$ is said to be $(a,b)$-interleaved if there exist morphisms $\alpha, \delta \colon F \to {T_a} G$ and $\beta, \gamma \colon G \to {T_b} F$ satisfying 
	\begin{enumerate}
		\renewcommand{\labelenumi}{$\mathrm{(\arabic{enumi})}$}
		\item $F \xrightarrow{\alpha} {T_a} G \xrightarrow{{T_a} \beta} {T_{a+b}} F$ is equal to $\tau_{0,a+b}(F) \colon F \to {T_{a+b}} F$ and
		\item $G \xrightarrow{\gamma} {T_b} F \xrightarrow{{T_b} \delta} {T_{a+b}} G$ is equal to $\tau_{0,a+b}(G) \colon G \to {T_{a+b}} G$.
	\end{enumerate}
	One can see that 
	\begin{center}
		$(a,b)$-isomorphic $\Rightarrow$ weakly $(a,b)$-isomorphic $\Rightarrow$ $(a,b)$-interleaved.
	\end{center}
	We also remark that the distance $d_{\cD(X)}$ in \cite{AI20} is defined by the relation ``$(a,b)$-interleaved" instead of ``$(a,b)$-isomorphic", and hence it is different from that in \cref{definition:distance_tamarkin}. 
	Later we will prove the main result in \cite{AI20} also holds for the modified $d_{\cD(X)}$ (see \cref{theorem:SQ_inequality}).
\end{remark}

The following proposition is slightly stronger than the similar results in the published version of \cite{AI20}. 

\begin{proposition}[{cf.\ \cite[Prop.~4.15]{AI20}}]\label{proposition:abisomhtpy}
	Let $I$ be an open interval containing the closed interval $[0,1]$ and
	$\cH \in \SD_{\{\tau \ge 0\}}(\bfk_{X \times \bR_t \times I})$.
	Assume that there exist continuous functions $f, g \colon I \to \bR_{\ge 0}$ satisfying
	\begin{equation}
		\MS(\cH) \subset T^*X \times \{(t,s;\tau,\sigma) \mid -f(s) \cdot \tau \le \sigma \le g(s) \cdot \tau \}.
	\end{equation}
	Then $\left(\cH|_{X \times \bR_t \times \{0\}},\cH|_{X \times \bR_t \times \{1\}} \right)$ is weakly $\left( \int_{0}^{1} g(s) ds+\varepsilon, \int_{0}^{1} f(s) ds +\varepsilon \right)$-isomorphic for any $\varepsilon \in \bR_{>0}$.
\end{proposition}

\begin{proof}
	The proof is similar to that of \cite[Prop.~4.15]{AI20}. 
	We only need to replace \cite[Lem.~4.14]{AI20} with \cref{lemma:torsion_to_weakisom}.
\end{proof}

\subsection{Sheaf quantization of Hamiltonian isotopies}

In this subsection, we first recall the existence and uniqueness result of a sheaf quantization of a Hamiltonian isotopy due to Guillermou--Kashiwara--Schapira~\cite{GKS}.

Let $M$ be a connected manifold without boundary and $I$ an open interval of $\bR$ containing the closed interval $[0,1]$. 
We say that a $C^\infty$-function $H=(H_s)_{s \in I} \colon T^*M \times I \to \bR$ is \emph{timewise compactly supported} if $\supp(H_s)$ is compact for any $s \in I$. 
A compactly supported \emph{Hamiltonian isotopy} is a flow of the Hamiltonian vector field of a timewise compactly supported $C^\infty$-function $H$. 
In this paper, the isotopy associated with $H$ is denoted by $\phi^H=(\phi^H_s)_{s \in I} \colon T^*M \times I \to T^*M$. 
Note that $(\phi^H_s)^{-1} = \phi^{\overline{H}}_s$ with $\overline{H}_s(p) \coloneqq -H_s(\phi^H_s(p))$.
Moreover, for two timewise compactly supported functions $H, H' \colon T^*M \times I \to \bR$, we have 
\begin{equation}
	\phi^H_s \circ \phi^{H'}_s = \phi^{H \sharp H'}_s, 
\end{equation}
where $(H \sharp H')_s(p) \coloneqq H_s(p)+H'_s((\phi^H_s)^{-1}(p))$.
In particular, for two timewise compactly supported functions $H, H' \colon T^*M \times I \to \bR$, 
\begin{equation}
	(\phi^H_s)^{-1} \circ \phi^{H'}_s = \phi^{\overline{H} \sharp H'}_s,
\end{equation}
where $(\overline{H} \sharp H')_s(p)=(H'-H)_s(\phi^H_s(p))$.

\begin{definition}\label{definition:homogeneous_ham}
	Let $\phi^H=(\phi^H_s)_{s\in I} \colon T^*M \times I \to T^*M$ be the compactly supported Hamiltonian isotopy associated with a timewise compactly supported function $H \colon T^*M \times I \to \bR$.
	\begin{enumerate}
		\item One defines $\wh{H} \colon \rT(M \times \bR_t) \times I \to \bR$ by 
		\begin{equation}
			\wh H((x,t;\xi,\tau),s)\coloneqq 
			\begin{cases}
				\tau H((x;\xi/\tau),s) & (\tau\neq 0)\\
				0 & (\tau =0)
			\end{cases}
		\end{equation}
		and $\wh \phi=(\wh \phi_s)_{s \in I} \colon \rT(M \times \bR_t) \times I \to \rT(M \times \bR_t)$ to be the homogeneous Hamiltonian flow of $\wh H$.
		\item One defines a conic Lagrangian submanifold $\Lambda_{\wh \phi}$ of $\rT(M \times \bR)^2 \times T^*I$ by 
        \begin{equation}
			\begin{aligned}
				\Lambda_{\wh{\phi}}
				\coloneqq 
				& 
				\left\{
				\left(
				\wh{\phi}((x,t;\xi,\tau),s), (x,t;-\xi,-\tau), (s;-\wh{H}(\wh{\phi}((x,t;\xi,\tau),s),s)) \right)
				\; \middle| \; \right. \\
				& \hspace{10pt} \left. (x,t;\xi,\tau) \in \rT(M \times \bR_t), 
				s \in I
				\right\}.
			\end{aligned}
		\end{equation}
	\end{enumerate}
\end{definition}

For a timewise compactly supported function $H \colon T^*M \times I \to \bR$, we also define $u=(u_s)_{s \in I} \colon T^*M \times I \to \bR$ by $u_s(p)=\int_0^s (H_{s'}-\alpha(X_{s'}))(\phi^H_{s'}(p)) ds'$, where $(X_s)_{s \in I}$ is the Hamiltonian vector field for $H$. 
Then $\wh \phi$ can be written as 
\begin{equation}
	\wh \phi_s (x,t;\xi,\tau)=(x',t+u_s(x;\xi/\tau);\xi',\tau),
\end{equation}
where $(x';\xi'/\tau)=\phi_s(x;\xi/\tau)$ for $\tau \neq 0$, and $\wh{\phi}_s(x,t;\xi,0)=(x,t;\xi,0)$.
Hereafter, we use the convention that $\tau H_s(\phi_s(x;\xi/\tau))=0$ and $u_s(x;\xi/\tau)=0$ when $\tau = 0$.
Moreover, we write $(x';\xi'/\tau)=\phi_s(x;\xi/\tau)$ also for $\tau=0$, in which case it is understood that $(x';\xi')=(x;\xi)$.
We have $du_s=\alpha-(\phi_s)^*\alpha$, which gives the following properties.

\begin{lemma}\label{lemma:properties_hamiltonian}
	If $\phi_1=\id_{T^*M}$, then $u_1\equiv 0$. 
\end{lemma}

The main theorem of \cite{GKS} is the following.

\begin{theorem}[{\cite[Thm.~3.7]{GKS}}]\label{thm:GKSmain}
	Let $\phi \colon T^*M \times I \to T^*M$ be a compactly supported Hamiltonian isotopy and $H \colon T^*M \times I \to \bR$ be a $C^\infty$-function with Hamiltonian flow $\phi$.
	Then there exists a unique simple object $\tl{K}^H \in \SD(\bfk_{(M \times \bR)^2 \times I})$ such that $\rMS(\tl{K}^H)=\Lambda_{\wh{\phi}}$ and $\tl{K}^H|_{(M \times \bR)^2 \times \{0\}} \simeq \bfk_{\Delta_{M \times \bR}}$.
\end{theorem}

Set $\tl{K}^H_{s}\coloneqq \tl{K}^H|_{(M \times \bR)^2 \times \{s\}} \in \SD(\bfk_{(M \times \bR)^2})$.
Note that $\mathring{\MS}(\tl{K}^H_s) \subset \Lambda_{\wh{\phi}} \circ T^*_sI$.
We also have 
\begin{equation}\label{equation:SQ_composition}
	\tl{K}^H_s \circ \tl{K}^{H'}_s \simeq \tl{K}^{H \sharp H'}_s.
\end{equation}
It is also proved by Guillermou--Schapira~\cite[Prop.~4.29]{GS14} that the composition with $\tl{K}^H_s$ defines a functor
\begin{equation}
	\tl{K}^H_s \circ (\ast) \colon \cD(M) \lto \cD(M).
\end{equation}

Define $q \colon (M \times \bR)^2 \times I \to M^2 \times \bR_t \times I, (x_1,t_1,x_2,t_2,s) \mapsto (x_1,x_2,t_1-t_2,s)$ and set 
\begin{equation}\label{eq:estimate_LambdaH}
	\begin{aligned}
		\Lambda'_H & \coloneqq q_\pi q_d^{-1}(\Lambda_{\wh{\phi}}) \\
		&=\left\{
		\left((x';\xi'),(x;-\xi), (u_s(x;\xi/\tau);\tau), (s;-\tau H_s(\phi_s(x;\xi/\tau)))\right)
		\; \middle| \; \right. \\
		& \hspace{25pt}  (x;\xi) \in T^*M, s \in I, \tau\in \bR,
		(x';\xi'/\tau)=\phi_s(x;\xi/\tau)  
		\Bigr\} \\
		&\subset \rT(M^2 \times \bR_t \times I).
	\end{aligned}
\end{equation} 
Then the inverse image functor $q^{-1}$ gives an equivalence (see \cite[Cor.~2.3.2]{Gu23})
\begin{equation}\label{eq:equivalence_GKS}
	\{ K \in \SD(\bfk_{M^2 \times \bR_t \times I}) \mid \rMS(K) = \Lambda'_H \} 
	\simto 
	\{ \tl{K} \in \SD(\bfk_{(M \times \bR)^2 \times I}) \mid \rMS(\tl{K}) = \Lambda_{\wh{\phi}} \}.
\end{equation}
Recall that we have a projector $P_l \colon \SD(\bfk_{M^2 \times \bR_t}) \to {}^\perp \SD_{\{\tau \le 0\}}(\bfk_{M^2 \times \bR_t})=\cD(M^2)$ and similarly for $\cD(M^2 \times I)$.

\begin{definition}
	Let $H \colon T^*M \times I \to \bR$ be a timewise compactly supported function.
	\begin{enumerate}
		\item One defines $K^H \in \SD(\bfk_{M^2 \times \bR_t \times I})$ to be the object such that $\rMS(K^H)=\Lambda'_H$ and $K^H|_{M^2 \times \bR_t \times \{0\}} \simeq \bfk_{\Delta_M \times \{0\}}$, that is, the object $K^H$ satisfying $q^{-1}K^H \simeq \tl{K}^H$. 
		One also sets $K^H_s \coloneqq K^H|_{M^2 \times \bR_t \times \{s\}}$ for simplicity.
		\item One defines $\cK^H_s \coloneqq P_l(K^H_s) \in \cD(M^2)$ and $\cK^H \coloneqq P_l(K^H) \in \cD(M^2 \times I)$. 
	\end{enumerate}
\end{definition}

Note that $\tl{K}^H_s \circ F \simeq K^H_s \bullet F$ for any $F \in \cD(M)$ and $s \in I$.
By the associativity, for any $F \in \cD(M)$ and $s \in I$, we have 
\begin{equation}
	\begin{aligned}
		K^H_s \bullet F 
		& \simeq K^H_s \bullet (\bfk_{M \times [0,\infty)} \star F) \\
		& \simeq (K^H_s \bullet_{M} \bfk_{M \times [0,\infty)}) \bullet F \\
		& \simeq (K^H_s \star \bfk_{M \times M \times [0,\infty)}) \bullet F 
		\simeq \cK^H_s \bullet F.
	\end{aligned}
\end{equation}

\section{Completeness of derived category of sheaves}\label{section:limit}

In this section, we prove the completeness of the derived category $\SD(\bfk_X)$ with respect to the pseudo-distance $d_\frakK$ associated with a thickening kernel $\frakK$.
If a category with a persistence structure admits any sequential colimit, then the category is complete with respect to the interleaving distance (Cruz~\cite{cruz2019metric} and Scoccola~\cite{scoccola2020locally}).
However, the derived category does not admit sequential colimits. 
Hence, we construct a limit object by using a homotopy colimit instead.
Let $X$ be a manifold throughout this section.

In \cref{lemma:torsion_to_weakisom}, we saw that for an exact triangle $F \to G \to H \toone$, if $H$ is $c$-torsion, then $(F,G)$ is weakly $(0,c)$-isomorphic. 
Conversely, we obtain \cref{proposition:interleaved_to_torsion_thickening} below, which is a key to our construction of limit objects.

\begin{lemma}\label{lemma:cone_torsion_thickening}
	Let $F \in \SD(\bfk_X)$ and $a \in \bR_{\ge 0}$ and consider the exact triangle 
	\begin{equation}
		\frakK_a \circ F \xrightarrow{\rho_{a,0}(F)} F \to \Cone(\rho_{a,0}(F)) \toone.
	\end{equation}
	Then $\Cone(\rho_{a,0}(F))$ is $2a$-torsion.
\end{lemma}

\begin{proof}
	Set $C \coloneqq \Cone(\rho_{a,0})$ and consider the following commutative diagram:
	\begin{equation}
		\begin{aligned}
			\xymatrix{
				\frakK_{3a} \circ F \ar[r] & \frakK_{2a} \circ F \ar[r] \ar[d] & \frakK_{2a} \circ C \ar[r] \ar[d] \ar@{-->}[ld] & \frakK_{3a} \circ F[1] \ar[r] \ar[d] & \frakK_{2a} \circ F[1] \ar@{=}[ld] \\
				& \frakK_{a} \circ F \ar[r] \ar[d] \ar@{=}[ld] & \frakK_{a} \circ C \ar[r] \ar[d] & \frakK_{2a} \circ F[1] \ar[d] & \\
				\frakK_{a} \circ F \ar[r] & F \ar[r] & C \ar[r] & \frakK_{a} \circ F[1]. & 
			}
		\end{aligned}
	\end{equation}
	The composite morphism $\frakK_{2a} \circ C \to \frakK_{3a} \circ F[1] \to \frakK_{2a} \circ F[1]$ is zero since $\frakK_{3a} \circ F \to \frakK_{2a} \circ F \to \frakK_{2a} \circ C \toone$ is an exact triangle.
	By the commutativity, the composite $\frakK_{2a} \circ C \to \frakK_{a} \circ C \to \frakK_{2a} \circ F[1]$ is also zero. 
	Hence, there exists a morphism $\frakK_{2a} \circ C \to \frakK_{a} \circ F$ that makes the lower triangle commutative. 
	Therefore, the morphism $\frakK_{2a} \circ C \to C$ factors the composite $\frakK_{a} \circ F \to F \to C$ and hence it is zero.
\end{proof}

\begin{proposition}\label{proposition:interleaved_to_torsion_thickening}
	Let $F, G \in \SD(\bfk_X)$ and assume that the pair $(F,G)$ is $(a,b)$-isomorphic. 
	Let $(\alpha \colon \frakK_a \circ F \to G, \beta \colon \frakK_b \circ G \to F)$ be an $(a,b)$-isomorphism for $(F,G)$ and consider the exact triangle
	\begin{equation}
		\frakK_a \circ F \to G \to \Cone(\alpha) \toone.
	\end{equation}
	Then $\Cone(\alpha)$ is $3(a+b)$-torsion.
\end{proposition}

\begin{proof}
	Consider the three exact triangles:
	\begin{equation}
		\begin{aligned}
			\xymatrix@R=10pt@C=30pt{
				\frakK_{a+b} \circ F \ar[r]^-{\frakK_b \circ \alpha} & \frakK_b \circ G \ar[r] & \frakK_b \circ \Cone(\alpha) \ar[r]^-{u} & \frakK_{a+b} \circ F[1], \\
				\frakK_{a+b} \circ F \ar[r]^-{\rho_{a+b,0}(F)} & F \ar[r] & \Cone(\rho_{a+b,0}(F)) \ar[r] & \frakK_{a+b} \circ F[1], \\
				\frakK_b \circ G \ar[r]^-{\beta} & F \ar[r] & \Cone(\beta) \ar[r] & \frakK_b \circ G[1].
			}
		\end{aligned}
	\end{equation}
	Note that $\beta \circ (\frakK_b \circ_X \alpha)=\rho_{a+b,0}(F)$. 
	By the octahedral axiom, we have the following commutative diagram, where the bottom row is also an exact triangle:
	\begin{equation}
		\begin{aligned}
			\xymatrix{
				\frakK_{a+b} \circ F \ar[r]^-{\frakK_b \circ \alpha} \ar@{=}[d] & \frakK_b \circ G \ar[r] \ar[d]^-{\beta} & \frakK_b \circ \Cone(\alpha) \ar[r]^-{u} \ar[d]^-{v} & \frakK_{a+b} \circ F[1] \ar@{=}[d] \\
				\frakK_{a+b} \circ F \ar[r]^-{\rho_{a+b,0}(F)} \ar[d]_-{\frakK_b \circ \alpha} & F \ar[r] \ar@{=}[d] & \Cone(\rho_{a+b,0}(F)) \ar[r] \ar[d] & \frakK_{a+b} \circ F[1] \ar[d] \\
				\frakK_b \circ G \ar[r]^-{\beta} \ar[d] & F \ar[r] \ar[d] & \Cone(\beta) \ar[r] \ar@{=}[d] & \frakK_b \circ G[1] \ar[d] \\
				\frakK_b \circ \Cone(\alpha) \ar[r] & \Cone(\rho_{a+b,0}(F)) \ar[r] & \Cone(\beta) \ar[r] & \frakK_b \circ \Cone(\alpha)[1].
			}
		\end{aligned}
	\end{equation}
	In particular, the morphism $u \colon \frakK_b \circ \Cone(\alpha) \to \frakK_{a+b} \circ F[1]$ factors through $\Cone(\rho_{a+b,0}(F))$. 
	Then the commutative diagram 
	\begin{equation}
		\begin{aligned}
			\xymatrix@C=60pt{
				\frakK_{2a+3b} \circ \Cone(\alpha) \ar[r]^-{\rho_{2a+3b,b}(\Cone(\alpha))} \ar[d]_-{\frakK_{2a+2b} \circ v} & \frakK_b \circ \Cone(\alpha) \ar[r]^-{u} \ar[d]_-{v} & \frakK_{a+b} \circ F[1] \\
				\frakK_{2a+2b} \circ \Cone(\rho_{a+b,0}(F)) \ar[r]_-{0} & \Cone(\rho_{a+b,0}(F)) \ar[ru] & 
			}
		\end{aligned}
	\end{equation}
	proves that the composite morphism in the first row is zero by \cref{lemma:cone_torsion_thickening}.
	Thus, we obtain a morphism $\frakK_{3a+3b} \circ \Cone(\alpha) \to \frakK_{a+b} \circ G$ that makes the following diagram commute, where the vertical arrows are the corresponding $\rho$'s:
	\begin{equation}
		\begin{aligned}
			\xymatrix@C=40pt{
				& & \frakK_{3a+3b} \circ \Cone(\alpha) \ar@{-->}[ld] \ar[d] &  \\ 
				& \frakK_{a+b} \circ G \ar[ld]_-{\frakK_a \circ \beta} \ar[d] \ar[r] & \frakK_{a+b} \circ \Cone(\alpha) \ar[r]^-{\frakK_b \circ u} \ar[d] & \frakK_{2a+b} \circ F[1] \\
				\frakK_a \circ F \ar[r]_-{\alpha} & G \ar[r] & \Cone(\alpha) \ar[r] & \frakK_a \circ F[1].
			}
		\end{aligned}
	\end{equation}
	Hence, the morphism $\rho_{3a+3b,0}(\Cone(\alpha))$ factors the composite morphism $\frakK_a \circ F \to G \to \Cone(\alpha)$ and thus it is zero.
\end{proof}

\begin{theorem}\label{theorem:limit_thickening}
	Let $(F_n)_{n \in \bN}$ be a sequence of objects in $\SD(\bfk_X)$ and assume that $F_n$ and $F_{n+1}$ are $a_n$-isomorphic with $\sum_{n} a_n < \infty$. 
	Set $a_{\ge n} \coloneqq \sum_{k \ge n} a_k$. 
	Then there exists an object $F_\infty \in \SD(\bfk_X)$ such that $(F_n,F_\infty)$ is $(2a_{\ge n}, 24a_{\ge n})$-isomorphic for any $n \in \bN$.
	In particular, $d_\frakK(F_n,F_\infty) \to 0 \ (n \to \infty)$.
\end{theorem}

\begin{proof}
	Set $G_n \coloneqq \frakK_{a_{\ge n}} \circ F_n$. 
	Then we have a morphism $\alpha_{n,m} \colon G_n \to G_m$ for $n \le m$ and get an inductive system $(G_n)_{n \in \bN}$. 
	Let $F_\infty \coloneqq \hocolim_{n} G_n \in \SD(\bfk_{X})$. 
	
	We fix $n$ and consider the cone $C_{n,m}$ of the morphism $\alpha_{n,m} \colon G_n \to G_m$.
	By composing $\beta_m$'s, we obtain $\beta_{m,n} \colon G_m \to G_n$.
	Since $(\alpha_{n,m}, \beta_{m,n})$ gives an $a_{\ge n}$-isomorphism between $G_n$ and $G_m$, the cone $C_{n,m}$ is $6a_{\ge n}$-torsion by \cref{proposition:interleaved_to_torsion_thickening}.
	Consider the following commutative diagram with solid arrows:
	\begin{equation}
		\begin{aligned}
			\xymatrix{
				G_n^{\oplus \bN} \ar[r] \ar[d] & \bigoplus_{m \ge n} G_m \ar[r] \ar[d] & \bigoplus_{m \ge n} C_{n,m} \ar[r] \ar@{..>}[d] & G_n^{\oplus \bN}[1] \ar[d] \\
				G_n^{\oplus \bN} \ar[r] \ar[d] & \bigoplus_{m \ge n} G_m \ar[r] \ar[d] & \bigoplus_{m \ge n} C_{n,m} \ar[r] \ar@{..>}[d] & G_n^{\oplus \bN}[1] \ar[d] \\
				\hocolim_m G_n \ar@{..>}[r] \ar[d] & \hocolim_{m \ge n} G_m \ar@{..>}[r] \ar[d] & H \ar@{..>}[r] \ar@{..>}[d] & \hocolim_m G_n[1] \ar[d] \\
				G_n^{\oplus \bN}[1] \ar[r] & \bigoplus_{m \ge n} G_m[1] \ar[r] & \bigoplus_{m \ge n} C_{n,m}[1] \ar[r] & G_n^{\oplus \bN}[2]. 
			}
		\end{aligned}
	\end{equation}
	Then by \cite[Exercise~10.6]{KS06}, the dotted arrows can be completed so that the right bottom square is anti-commutative, all the other squares are commutative, and all the rows and all the columns are exact triangles. 
	Since $\bigoplus_{m \ge n} C_{n,m}$ is $6a_{\ge n}$-torsion, $H$ is $12a_{\ge n}$-torsion.
	Noticing that $\hocolim_m G_n \simeq G_n$ and $\hocolim_{m \ge n} G_m \simeq F_\infty$, by \cref{lemma:torsion_to_weakisom} $(G_n,F_\infty)$ is weakly $(0,12a_{\ge n})$-isomorphic. 
	Since we set $G_n = \frakK_{a_{\ge n}} \circ F_n$, we find that $(F_n,F_\infty)$ is weakly $(a_{\ge n},12a_{\ge n})$-isomorphic, which implies that $(F_n,F_\infty)$ is $(2a_{\ge n},24a_{\ge n})$-isomorphic (see \cref{remark:distances}).
\end{proof}

\begin{remark}
	One can prove a similar completeness result in a more general setting, i.e., for a triangulated category with a persistence structure (cf.\ persistence triangulated category by Biran--Cornea--Zhang~\cite{biran2021triangulation}). 
	Here we do not go into details.
\end{remark}

\begin{corollary}\label{corollary:limit_Tamarkin}
	The derived category of sheaves $\SD(\bfk_X)$ is complete with respect to the pseudo-distance $d_\frakK$. 
	In particular, the Tamarkin category $\cD(X)$ is complete with respect to the pseudo-distance $d_{\cD(X)}$.
\end{corollary}

\begin{proof}
	For the latter claim, it suffices to show that for a Cauchy sequence $(F_n)_{n \in \bN}$ in the Tamarkin category $\cD(X)$, the limit object $F_\infty$ constructed in \cref{theorem:limit_thickening} is also in $\cD(X)$.
	By construction and \cref{lem:shift}, after taking a subsequence, we have $F_\infty = \hocolim_m T_{-a_{\ge n}} F_n$, where $(a_{\ge n})_{n \in \bN}$ is as in \cref{theorem:limit_thickening}.
	Set $G_n \coloneqq T_{-a_{\ge n}} F_n$.
	Since the functor $\bfk_{X \times [0,\infty)} \star (\ast)$ is defined as the composite of left adjoint functors, it commutes with direct sums. 
	Since each $G_n$ is an object of the left orthogonal ${}^{\bot} \SD_{\{\tau \le 0\}}(\bfk_{X \times \bR_t})=\cD(X)$, in the following commutative diagram the first and the second vertical arrows are isomorphisms:
	\begin{equation}
    \begin{aligned}
		\xymatrix{
			\bfk_{X \times [0,\infty)} \star \bigoplus_{n} G_n \ar[r] \ar[d]^-{\rotatebox{90}{$\sim$}} & \bfk_{X \times [0,\infty)} \star \bigoplus_{n} G_n \ar[r] \ar[d]^-{\rotatebox{90}{$\sim$}} & \bfk_{X \times [0,\infty)} \star F_\infty \ar[r]^-{+1} \ar[d] & \\
			\bigoplus_{n} G_n \ar[r] & \bigoplus_{n} G_n \ar[r] & F_\infty \ar[r]^-{+1} &.     
		}
    \end{aligned}
	\end{equation}
	Thus, we have an isomorphism $\bfk_{X \times [0,\infty)} \star F_\infty \simto F_\infty$ and find that $F_\infty \in \cD(X)$.
\end{proof}

\begin{remark}
	As mentioned in \cref{remark:distances}, $d_\frakK$ and $\dist_\frakK$ are equivalent. 
	Hence $\SD(\bfk_X)$ is also complete with respect to $\dist_\frakK$.
\end{remark}

\section{Sheaf quantization of Hamiltonian diffeomorphisms and homeomorphisms}\label{section:SQ_diffeo_homeo}

In this section, we give a refined Hamiltonian stability result, which state the distance between sheaf quantizations of Hamiltonian diffeomorphisms is at most the Hofer distance.  
By using the stability result, we also construct a sheaf quantization of a Hamiltonian homeomorphism.

\subsection{Hamiltonian stability theorem}

In this subsection, we give a generalization of our previous result \cite[Thm.~4.16]{AI20}.

For a timewise compactly supported function $H \colon T^*M \times I \to \bR$, we define 
\begin{equation}
	\begin{aligned}
		E_+(H)
		& \coloneqq 
		\int_0^1 \max_{p \in T^*M} H_s(p) ds, 
		\qquad 
		E_-(H)
		\coloneqq 
		-\int_0^1 \min_{p \in T^*M} H_s(p) ds, \\
		\| H \|_{\mathrm{osc}}
		& \coloneqq 
		E_+(H)+E_-(H)
		=
		\int_0^1 \left(\max_{p \in T^*M} H_s(p) - \min_{p \in T^*M} H_s(p) \right) ds.
	\end{aligned}
\end{equation}
When $M=\pt$, we need to modify $E_+(H)$ and $E_-(H)$ so that they are non-negative.

\begin{theorem}\label{theorem:SQ_inequality}
	Let $H \colon T^*M \times I \to \bR$ be a timewise compactly supported function. 
	Then $(\cK^0_1,\cK^H_1)$ is $(E_-(H)+\varepsilon, E_+(H)+\varepsilon)$-isomorphic for any $\varepsilon \in \bR_{>0}$, where $0$ denotes the zero function on $T^*M \times I$.
	In particular, $d_{\cD(M^2)}(\cK^0_1,\cK^H_1) \le \|H\|_{\mathrm{osc}}$.
\end{theorem}

\begin{proof}
	Let $\varepsilon>0$ be an arbitrary positive number.
	We apply \cref{proposition:abisomhtpy} to $\cK^H$.
	Then by \eqref{eq:estimate_LambdaH}, we find that $(\cK^0_1,\cK^H_1)=(\cK^H_0,\cK^H_1)$ is weakly $(E_-(H)+\varepsilon, E_+(H)+\varepsilon)$-isomorphic.
	It is enough to show that $(\cK^0_1,\cK^H_1)$ is $(E_-(H)+\varepsilon, E_+(H)+\varepsilon)$-isomorphic. 
	
	We set $a \coloneqq E_-(H)+\varepsilon$ and $b \coloneqq E_+(H)+\varepsilon$.
	Then, by definition, there exist morphisms $\alpha, \delta \colon \cK^0_1 \to T_a \cK^H_1$ and $\beta, \gamma \colon \cK^H_1 \to T_b \cK^0_1$ such that
	\begin{enumerate}
		\renewcommand{\labelenumi}{$\mathrm{(\arabic{enumi})}$}
		\item $\cK^0_1 \xrightarrow{\alpha} {T_a} \cK^H_1 \xrightarrow{{T_a} \beta} {T_{a+b}} \cK^0_1$ is equal to $\tau_{0,a+b}(\cK^0_1) \colon \cK^0_1 \to {T_{a+b}} \cK^0_1$,
		\item $\cK^H_1 \xrightarrow{\gamma} {T_b} \cK^0_1 \xrightarrow{{T_b} \delta} {T_{a+b}} \cK^H_1$ is equal to $\tau_{0,a+b}(\cK^H_1) \colon \cK^H_1 \to {T_{a+b}} \cK^H_1$, and
		\item $\tau_{a,2a}(\cK^H_1) \circ \alpha=\tau_{a,2a}(\cK^H_1) \circ \delta$ and $\tau_{b,2b}(\cK^0_1) \circ \beta=\tau_{b,2b}(\cK^0_1) \circ \gamma$.
	\end{enumerate}
	
	Now we let $\Tor$ be the full triangulated subcategory of $\cD(M^2)$ consisting of torsion objects $\{F \mid d_{\cD(M^2)}(F, 0)<\infty\}$. 
	Then, by \cite[Prop.~6.7]{GS14}, the Hom set of the localized category $\cD(M^2)/\Tor$ is computed as 
	\begin{equation}
		\Hom_{\cD(M^2)/\Tor}(F,G) \simeq \varinjlim_{c\to \infty}\Hom_{\cD(M^2)}(F,T_cG). 
	\end{equation}
	For the objects $\cK^0_1, \cK^H_1$ and $d \in \bR$, we have 
	\begin{equation}
		\Hom_{\cD(M^2)}(\cK^H_1,T_d \cK^H_1) \simeq 
		\Hom_{\cD(M^2)}(\cK^0_1,T_d \cK^0_1) \simeq 
		\begin{cases}
			\bfk & (d \ge 0) \\
			0 & (d <0).
		\end{cases}
	\end{equation}
	Hence, $\Hom_{\cD(M^2)/\Tor}(\cK^0_1,\cK^0_1) \simeq \Hom_{\cD(M^2)/\Tor}(\cK^H_1,\cK^H_1) \simeq \bfk$ and the canonical morphism
	\begin{equation}
		\Hom_{\cD(M^2)}(\cK^H_1,T_d \cK^H_1) \to \Hom_{\cD(M^2)/\Tor}(\cK^H_1,\cK^H_1), \alpha' \mapsto \overline{\alpha'}
	\end{equation}
	is injective for $d \ge 0$. 
	
	By the condition~(3), we have 
	\begin{align}
		\overline{\alpha}=\overline{\delta} & \in \Hom_{\cD(M^2)/\Tor}(\cK^0_1,\cK^H_1) \\
		\overline{\beta}=\overline{\gamma} & \in \Hom_{\cD(M^2)/\Tor}(\cK^H_1,\cK^0_1).
	\end{align}
	Hence, through the isomorphism $\Hom_{\cD(M^2)/\Tor}(\cK^H_1,\cK^H_1) \simeq \bfk$, we get
	\begin{equation}
		1 = \overline{\delta} \circ \overline{\gamma} = \overline{\alpha} \circ \overline{\beta} \in \Hom_{\cD(M^2)/\Tor}(\cK^H_1,\cK^H_1),
	\end{equation}
	by the condition~(2).
	Since $T_b \alpha \circ \beta \in \Hom_{\cD(M^2)}(\cK^H_1,T_{a+b} \cK^H_1)$ is sent to $\overline{\alpha} \circ \overline{\beta}$ and $\tau_{0,a+b}(\cK^H_1)$ is sent to $1$, by the injectivity we obtain $T_b \alpha \circ \beta=\tau_{0,a+b}(\cK^H_1)$. 
	Thus, combining this with the condition~(1), we find that the pair $(\alpha,\beta)$ gives an $(a,b)$-isomorphism for the pair $(\cK^0_1,\cK^H_1)$, which completes the proof.
\end{proof}

For two timewise compactly supported functions $H,H' \colon T^*M \times I \to \bR$, by \eqref{eqn:bullet_distance} and \eqref{equation:SQ_composition}, we have
\begin{equation}\label{equation:inequality_SQ}
	d_{\cD(M^2)}(\cK^H_1,\cK^{H'}_1)
	= 
	d_{\cD(M^2)}(\cK^0_1, \cK^{\overline{H} \sharp H'}_1)
	\le \| H -H'\|_{\mathrm{osc}}.
\end{equation}

\subsection{Sheaf quantization of Hamiltonian diffeomorphisms}

In this subsection, we investigate sheaf quantization of Hamiltonian diffeomorphisms.
We keep the symbols $M$ for a connected manifold without boundary and $I$ for an open interval containing $[0,1]$. 
First, we prove the following, an analogue of \cref{theorem:SQ_inequality} in the derived category $\SD(\bfk_{M^2 \times \bR_t})$ not in the Tamarkin category $\cD(M^2)$.

\begin{proposition}\label{theorem:energy_bound_derived}
	Let $H \colon T^*M \times I \to \bR$ be a timewise compactly supported function. 
	Then, one has an inequality 
	\begin{equation}
		d_{M^2 \times \bR_t}(K^0_1,K^H_1) \le 4 \int_0^1 \|H_s\|_{\infty} ds \le 4 \|H\|_{\mathrm{osc}}.
	\end{equation}
\end{proposition}

Similarly to \eqref{equation:inequality_SQ}, for two timewise compactly supported functions $H,H' \colon T^*M \times I \to \bR$, we have 
\begin{equation}\label{equation:inequality_SQ_derived}
	d_{M^2 \times \bR_t}(K^{H}_1,K^{H'}_1) \le 4 \|H-H'\|_{\mathrm{osc}}.
\end{equation}

To prove the proposition, we prepare some lemmas.
Let $X$ be a manifold and let $q \colon X \times \bR_t \to \bR_t$ denote the projection.

\begin{lemma}\label{lem:conv-hom-MS-const}
	Let $F \in \SD_{\{ \tau \le 0 \}}(\bfk_{X \times \bR_t})$ and $G \in \SD_{\{ \tau \ge 0 \}}(\bfk_{X \times \bR_t})$.
	Then there exist $H, H' \in \SD(\bfk_X)$ such that $F \star G \simeq H \boxtimes \bfk_{\bR_t}$ and $\cHom^\star(F,G) \simeq H' \boxtimes \bfk_{\bR_t}$.
\end{lemma}

\begin{proof}
	By \cref{lem:conv-hom-MS}, we find that $\MS(F \star G) \subset \{ \tau =0\}$ and $\MS(\cHom^\star(F,G)) \subset \{ \tau=0\}$, which imply the result.
\end{proof}

For $a,b \in \bR_{\ge 0}$, we set 
\begin{equation}
	D(a,b)\coloneqq 
	\bigcup_{-a \le c \le b} \{ (\tau,\sigma) \in \bR^2 \mid \sigma = c \cdot \tau \}.
\end{equation}
Note that $D(a,b)$ is a closed cone in $\bR^2$, which is not necessarily convex. 
See \cref{figure:Dab}.
We set $D^+(a,b)\coloneqq D(a,b) \cap \{(\tau,\sigma) \mid \tau \ge 0 \}$ and  $D^-(a,b)\coloneqq D(a,b) \cap \{(\tau,\sigma) \mid \tau \le 0 \}$.
\begin{figure}[H]
	\begin{center}
		\begin{tikzpicture}
			\draw [gray, fill=lightgray] (-3,-1.5) -- (-3,1) -- (0,0);
			\draw [gray, fill=lightgray] (3,1.5) -- (3,-1) -- (0,0);
			\draw [very thick] (-3,-1.5)--(3,1.5);
			\draw [very thick] (-3,1)--(3,-1);
			\node at (0.2,0.35){$O$};
			\draw [->] (0,-1.6) -- (0,1.6) node[left] {$\sigma$};
			\draw [->] (-3.5,0) -- (3.5,0) node[below] {$\tau$};
			\node at (3,1.5) [above] {$\sigma=b \cdot \tau$};
			\node at (3,-1) [below] {$\sigma=-a \cdot \tau$};
		\end{tikzpicture}
		\caption{$D(a,b)$}\label{figure:Dab}
	\end{center}
\end{figure}

\begin{lemma}\label{lem:maxab}
	Let $F \to G \to H \toone$ be an exact triangle in $\SD(\bfk_{X \times \bR_t})$
	and $a,b \in \bR_{\ge 0}$.
	Assume 
	\begin{enumerate}
		\renewcommand{\labelenumi}{$\mathrm{(\arabic{enumi})}$}
		\item $F \in \SD_{\{ \tau \le 0\}}(\bfk_{X \times \bR_t})$ and $F$ is $a$-torsion,
		\item $G \in \SD_{\{ \tau \ge 0\}}(\bfk_{X \times \bR_t})$ and $G$ is $b$-torsion.
	\end{enumerate}
	Then, $\RHom(F,G) \simeq 0$.
	In particular, $H \simeq G \oplus F[1]$ is $\max(a,b)$-torsion.
\end{lemma}

\begin{proof}
	By \cref{lem:conv-hom-MS-const}, there exists $H' \in \SD(\bfk_X)$ such that $\cHom^\star(F,G) \simeq H' \boxtimes \bfk_{\bR_t}$.
	Moreover by the isomorphism ${T_b} \cHom^\star(F,G) \simeq \cHom^\star(F,{T_b} G)$, we find that $\cHom^\star(F,G)$ is $b$-torsion.
	Hence we have $\cHom^\star(F,G) \simeq 0$ and  
	\begin{equation}\label{eq:hom-star}
		\begin{aligned}
			\RHom(F,G) 
			& \simeq 
			\RHom(\bfk_{X \times \{0\}} \star F, G) \\
			& \simeq 
			\RHom(\bfk_{X \times \{0\}}, \cHom^\star(F,G)) 
			\simeq 0,
		\end{aligned}
	\end{equation}
	which proves the result.
\end{proof}

Next, we give a microlocal cut-off result, which we use to reduce the problem to \cref{proposition:abisomhtpy}.

\begin{lemma}\label{lem:cutoff}
	Let $\cH \in \SD(\bfk_{X \times \bR_t \times \bR})$ and assume that there exist $a,b \in \bR_{\ge 0}$ such that 
	\begin{equation}
		\MS(\cH) \subset T^*X \times (\bR_t \times \bR) \times D(a,b).
	\end{equation}
	Then there exists an exact triangle $\cH^- \to \cH^+ \to \cH \toone$ in $\SD(\bfk_{X \times \bR_t \times \bR})$ such that $\MS(\cH^-) \subset T^*X \times (\bR_t \times \bR) \times D^-(a,b)$ and $\MS(\cH^+) \subset T^*X \times (\bR_t \times \bR) \times D^+(a,b)$.
\end{lemma}

\begin{proof}
	Let $\lambda \coloneqq \{(t,s) \mid t \ge 0, s=0\}$.
	Then we get an exact triangle 
	\begin{equation}
		\bfk_{X \times (\lambda \setminus \{0\})} \star \cH
		\lto \bfk_{X \times \lambda} \star \cH 
		\lto \cH \toone
	\end{equation}
	with $\MS(\bfk_{X \times \lambda} \star \cH) \subset T^*X \times (\bR_t \times \bR) \times D(a,b) \cap \{(\tau,\sigma) \mid \tau \ge 0 \}$.
	Moreover $\bfk_{X \times \lambda} \star F \to F$ is an isomorphism on $T^*X \times (\bR_t \times \bR) \times \Int(\lambda^\circ)$, where $\lambda^\circ$ denotes the polar cone of $\lambda$. 
	Thus, we conclude that $\MS(\bfk_{\lambda \setminus \{0\}} \star \cH) \subset T^*X \times (\bR_t \times \bR) \times (D(a,b) \setminus \Int(\lambda^\circ)) = T^*X \times (\bR_t \times \bR) \times D^-(a,b)$.
\end{proof}

The following is a variant of \cite[Prop.~4.3]{AI20}.

\begin{proposition}\label{prop:distance-constant-bicone}
	Let $\cH \in \SD(\bfk_{X \times \bR_t \times I})$ and $s_1 < s_2 \in I$.
	Assume that there exist $a,b \in \bR_{\ge 0}$ and $r \in \bR_{>0}$ such that 
	\begin{equation}
		\MS(\cH) \cap \pi^{-1}(X \times \bR_t \times (s_1-r,s_2+r)) 
		\subset T^*X \times (\bR_t \times I) \times D(a,b).
	\end{equation}
	\begin{enumerate}
		\item Let $q \colon X \times \bR_t \times I \to X \times \bR_t$ be the projection.
		Then $\rR q_*\cH_{X \times \bR_t \times (s_1,s_2]}$ and \linebreak $\rR q_*\cH_{X \times \bR_t \times [s_1,s_2)}$ are $(\max(a,b)(s_2-s_1)+\varepsilon)$-torsion for any $\varepsilon \in \bR_{>0}$.
		\item One has $d_{X \times \bR_t}(\cH|_{X \times \bR_t \times \{s_1\}},\cH|_{X \times \bR_t \times \{s_2\}}) \le 4 \max(a,b)(s_2-s_1)$.
	\end{enumerate}    
\end{proposition}

\begin{proof}
	(i)
	Choose a diffeomorphism $\varphi \colon (s_1-r,s_2+r) \simto \bR$ satisfying 
	$\varphi|_{[s_1,s_2]}= \id_{[s_1,s_2]}$ and $d\varphi(s) \ge 1$ for any $s \in (s_1-r,s_2+r)$.
	Set $\Phi \coloneqq \id_{X \times \bR_t} \times \varphi \colon X \times \bR_t \times (s_1-r,s_2+r) \simto X \times \bR_t \times \bR$ and $\cH' \coloneqq \Phi_* \cH \in \SD(\bfk_{X \times \bR_t \times \bR})$.
	Then by the assumption on $\varphi$, we have 
	\begin{equation}
		\MS(\cH') \subset T^*X \times (\bR_t \times \bR) \times D(a,b)
	\end{equation}
	and $\cH'|_{X \times \bR_t \times [s_1,s_2]} \simeq \cH|_{X \times \bR_t \times [s_1,s_2]}$.
	Hence, we may assume $I=\bR$ from the beginning.
	
	Applying \cref{lem:cutoff}, we have an exact triangle $\cH^- \to \cH^+ \to \cH \toone$ in $\SD(\bfk_{X \times \bR_t \times \bR})$ with $\MS(\cH^-) \subset T^*X \times (\bR_t \times \bR) \times D^-(a,b)$ and $\MS(\cH^+) \subset T^*X \times (\bR_t \times \bR) \times D^+(a,b)$.
	By \cite[Prop.~4.3]{AI20}, $\rR q_*\cH^+_{X \times \bR_t \times (s_1,s_2]}$ is $(b(s_2-s_1)+\varepsilon)$-torsion. 
	Similarly we find that $\rR q_*\cH^-_{X \times \bR_t \times (s_1,s_2]}$ is $(a(s_2-s_1)+\varepsilon)$-torsion.
	Here we have an exact triangle 
	\begin{equation}
		\rR q_*\cH^-_{X \times \bR_t \times (s_1,s_2]} \lto \rR q_*\cH^+_{X \times \bR_t \times (s_1,s_2]} \lto \rR q_*\cH_{X \times \bR_t \times (s_1,s_2]} \toone
	\end{equation}
	with $\rR q_*\cH^+_{X \times \bR_t \times (s_1,s_2]} \in \SD_{\{ \tau \ge 0 \}}(\bfk_{X \times \bR_t})$ and $\rR q_*\cH^-_{X \times \bR_t \times (s_1,s_2]} \in \SD_{\{ \tau \le 0 \}}(\bfk_{X \times \bR_t})$.
	Hence, by \cref{lem:maxab} we find that $\rR q_*\cH_{X \times \bR_t \times (s_1,s_2]}$ is  $(\max(a,b)(s_2-s_1)+\varepsilon)$-torsion.
	The proof for the other case is similar.
	
	\noindent (ii) 
	We have the following two exact triangles 
	\begin{align}
		& \rR q_*\cH_{X \times \bR_t \times (s_1,s_2]} \lto 
		\rR q_* \cH_{X \times \bR_t \times [s_1,s_2]} \lto 
		\cH|_{X \times \bR_t \times \{s_1\}} \toone, \\
		& \rR q_*\cH_{X \times \bR_t \times [s_1,s_2)} \lto 
		\rR q_* \cH_{X \times \bR_t \times [s_1,s_2]} \lto 
		\cH|_{X \times \bR_t \times \{s_2\}} \toone. 
	\end{align}
	Hence, by the result of (i) and \cref{lemma:torsion_to_weakisom}, the two pairs $(\rR q_* \cH_{X \times \bR_t \times [s_1,s_2]}, \cH|_{X \times \bR_t \times \{s_1\}})$ and $(\rR q_* \cH_{X \times \bR_t \times [s_1,s_2]}, \cH|_{X \times \bR_t \times \{s_2\}})$ are weakly $(0,(\max(a,b)(s_2-s_1)+\varepsilon))$-isomorphic. 
	Hence, the result follows from the triangle inequality.
\end{proof}

The following proposition is a variant of \cref{proposition:abisomhtpy}.

\begin{proposition}\label{thm:distance}
	Let $I$ be an open interval containing the closed interval $[0,1]$ and
	$\cH \in \SD(\bfk_{X \times \bR_t \times I})$.
	Assume that there exist continuous functions $f, g \colon I \to \bR_{\ge 0}$ satisfying
	\begin{equation}
		\MS(\cH) \subset T^*X \times \{(t,s;\tau,\sigma) \mid (\tau,\sigma) \in D(f(s),g(s)) \}.
	\end{equation}
	Then $d_{X \times \bR_t}(\cH|_{X \times \bR_t \times \{0\}},\cH|_{X \times \bR_t \times \{1\}}) \le 4 \int_{0}^{1} \max(f,g)(s) ds$.
\end{proposition}

\begin{proof}
	We can apply an argument similar to \cite[Prop.~4.15]{AI20}. 
	We only need to replace \cite[Prop.~4.3]{AI20} with \cref{prop:distance-constant-bicone}.
\end{proof}

\begin{proof}[Proof of \cref{theorem:energy_bound_derived}]
	The result follows from \eqref{eq:estimate_LambdaH} and \cref{thm:distance}.
\end{proof}

\begin{remark}
	One could prove an inequality 
	\begin{equation}
		d_{M^2 \times \bR_t}(K^0_1,K^H_1) \le 2 \int_0^1 \|H_s\|_{\infty} ds,
	\end{equation}
	which is stronger than \cref{theorem:energy_bound_derived}. 
	Indeed, by the proof of \cref{prop:distance-constant-bicone}, we have proved that under the assumption of \cref{thm:distance} $(\cH|_{X \times \bR_t \times \{0\}},\cH|_{X \times \bR_t \times \{1\}})$ is weakly $(\int_{0}^{1} \max(f,g)(s) ds+\varepsilon,\int_{0}^{1} \max(f,g)(s) ds+\varepsilon)$-isomorphic for any $\varepsilon \in \bR_{>0}$. 
	Hence, we find that $(K^0_1,K^H_1)$ is weakly $(\int_0^1 \|H_s\|_{\infty} ds+\varepsilon, \int_0^1 \|H_s\|_{\infty} ds+\varepsilon)$-isomorphic for any $\varepsilon \in \bR_{>0}$. 
	It remains to apply an argument similar to the proof of \cref{theorem:SQ_inequality}. 
	However, we do not need this stronger inequality in this paper. 
\end{remark}

It is proved in \cite[Prop.~4.3]{Zhang20} that the restriction of the sheaf quantization $K^H$ to $s=1$ depends only on the relative homotopy class of the path $[s \mapsto \phi^H_s]$.
Now we prove the following stronger result, which claims that the restriction depends only on the time-1 map.

\begin{proposition}\label{proposition:dependence_time_one}
	Let $H \colon T^*M \times I \to \bR$ be a timewise compactly supported function.
	Then the objects $\tl{K}^H_1,K^H_1,\cK^H_1$ are determined by the time-1 map $\phi^H_1$. 
\end{proposition}

\begin{proof}
	We shall prove $K^H_1=K^{H'}_1$ assuming that $\phi^{H'}_1=\phi^H_1$. 
	By \eqref{equation:SQ_composition}, it suffices to show that $K^{\overline{H} \sharp H'}_1\simeq \bfk_{\Delta_M\times \{0\}}$. 
	Hence, we may assume that $H'\equiv 0$ and $\phi^H_1=\id_{T^*M}$. 
	
	Since $u_s\equiv 0$ by \cref{lemma:properties_hamiltonian}, we find that $\rMS(K^H_1)=\rT_{\Delta_M\times \{0\}}(M^2 \times \bR_t)$ and $K^H_1$ is simple along the subset. 
	Moreover, since $K^H_0\simeq \bfk_{\Delta_M\times \{0\}}$ and $u$ is compactly supported and hence bounded, $K^H|_{M^2 \times \{ R\}\times I}$ and $K^H|_{M^2 \times \{ -R\}\times I}$ are $0$ for sufficiently large $R$. 
	Hence, there exists a rank one local system $\cL$ on $M$ such that $K^H_1 \simeq {\delta_M}_*\cL \boxtimes \bfk_{\{0\}}[m]$, where $m$ is some integer. 
	By \cref{theorem:energy_bound_derived}, we have 
	\begin{equation}
		d_{M^2 \times \bR_t}(\bfk_{\Delta_M\times \{0\}}, {\delta_M}_*\cL \boxtimes \bfk_{\{0\}}[m])=  d_{M^2 \times \bR_t}(K^0_1,K^H_1) \le 4 \|H\|_{\mathrm{osc}}<\infty.
	\end{equation}
	By restricting to $\{(x,x)\} \times \bR \subset M^2 \times \bR_t$ for some $x \in M$, we obtain $d_{\bR_t}(\bfk_{\{0\}}, \bfk_{\{0\}}[m]) < \infty$ and find that $m=0$.
	Then, we have 
	\begin{equation}
		\RG(M;\cL) 
		\simeq \RG(M^2 \times \bR_t;K^H_1) 
		\simeq \RG(M^2 \times \bR_t;K^0_1)
		\simeq \RG(M;\bfk_M).
	\end{equation}
	In particular, $H^0(M;\cL) \simeq H^0(M;\bfk_M)$, which implies that $\cL$ is trivial.  
\end{proof}

The proposition above shows the well-definedness in the following definition. 

\begin{definition}
	\begin{enumerate}
		\item A diffeomorphism $\varphi \colon T^*M \to T^*M$ is said to be a compactly supported \emph{Hamiltonian diffeomorphism} if it is the time-1 map of some compactly supported Hamiltonian isotopy $\phi^H$, that is $\varphi=\phi^H_1$. 
		The set of compactly supported Hamiltonian diffeomorphisms is denoted by $\Ham_c(T^*M,\omega)$.
		\item The \emph{Hofer metric} between Hamiltonian diffeomorphisms is defined by
		\begin{equation}
			d_H(\varphi,\varphi')\coloneqq \inf \left\{ \|H\|_{\mathrm{osc}} \relmid \phi^H_1=\varphi^{-1}\varphi'\right\}
		\end{equation}
		for $\varphi,\varphi'\in  \Ham_c(T^*M, \omega)$. 
		\item For $\varphi \in \Ham_c(T^*M, \omega)$, one sets $\tl{K}^\varphi\coloneqq \tl{K}^H_1, K^\varphi\coloneqq K^H_1$, and $\cK^\varphi\coloneqq \cK^H_1$, where $H$ is any timewise compactly supported function with $\varphi=\phi^H_1$.  
	\end{enumerate}
\end{definition}

By \eqref{equation:inequality_SQ} and \eqref{equation:inequality_SQ_derived}, we have the following.

\begin{theorem}\label{theorem:stability_hofer}
	For $\varphi, \varphi' \in \Ham_c(T^*M,\omega)$, one has 
	\begin{equation}
		\begin{aligned}
			d_{\cD(M^2)}(\cK^\varphi,\cK^{\varphi'}) & \le d_H(\varphi,\varphi'), \\
			d_{M^2 \times \bR_t}(K^\varphi, K^{\varphi'}) & \le 4 d_H(\varphi, \varphi').
		\end{aligned}
	\end{equation}
\end{theorem}

\subsection{Sheaf quantization of Hamiltonian homeomorphisms}

In this subsection, we construct a sheaf quantization of a limit of Hamiltonian diffeomorphisms with respect to the Hofer metric.

\begin{proposition}\label{prop:SQ_limit_derived}
	Let $(\varphi_n)_{n \in \bN} \subset \Ham_c(T^*M,\omega)$ be a sequence of compactly supported Hamiltonian diffeomorphisms and $K_n \coloneqq K^{\varphi_n} \in \SD(\bfk_{M^2 \times \bR_t})$ the sheaf quantization of $\varphi_n$.
	Assume that it is a Cauchy sequence with respect to the Hofer metric $d_H$.
	Then there exists an object $K_\infty \in \SD(\bfk_{M^2 \times \bR_t})$ such that $d_{M^2 \times \bR_t}(K_n,K_\infty) \to 0 \ (n \to \infty)$.
	Moreover, such an object is unique up to isomorphism, and the endofunctor $K_\infty \bullet (\ast)$ on $\SD(\bfk_{M \times \bR_t})$ gives an equivalence of categories. 
\end{proposition}

\begin{proof}
	By \eqref{equation:inequality_SQ_derived}, we have 
	\begin{equation}
		d_{M^2 \times \bR_t}(K_n,K_m) \le 4 d_H(\varphi_n,\varphi_m),
	\end{equation}
	which proves that $(K_n)_{n \in \bN}$ is a Cauchy sequence in $\SD(\bfk_{M^2 \times \bR_t})$ with respect to $d_{M^2 \times \bR_t}$.
	Hence, \cref{corollary:limit_Tamarkin} shows the existence of a limit object.
	
	Let $K'_\infty$ be another limit object that satisfies $d_{M^2 \times \bR_t}(K_\infty, K'_\infty)=0$. 
	There exists $\overline{K}_n\in \SD(\bfk_{M^2 \times \bR_t})$ such that $\overline{K}_n\bullet K_n\simeq K_n\bullet \overline{K}_n\simeq \bfk_{\Delta_M\times \{0\}}$ for each $n$. 
	Then we find that the sequence $(\overline{K}_n)_{n \in \bN}$ is also Cauchy, and we can take a limit object $\overline{K}_\infty$. 
	The Cauchy sequence $(\overline{K}_n \bullet K_n)_{n \in \bN}$ converges to both $\overline{K}_\infty \bullet K_\infty$ and $\bfk_{\Delta_M\times \{0\}}$. Hence we have $d_{M^2 \times \bR_t}(\overline{K}_\infty \bullet K_\infty, \bfk_{\Delta_M\times \{0\}})=0$. 
	Similarly, we have $d_{M^2 \times \bR_t}(K'_\infty \bullet \overline{K}_\infty, \bfk_{\Delta_M\times \{0\}})=0$. 
	By \cref{lemma:zero_distance} below, we have
	$\overline{K}_\infty \bullet K_\infty\simeq \bfk_{\Delta_M\times \{0\}}\simeq K'_\infty \bullet \overline{K}_\infty$. 
	Hence, we conclude that $K'_\infty \simeq K'_\infty \bullet \overline{K}_\infty\bullet K_\infty\simeq K_\infty$ and $\overline{K}_\infty \bullet (\ast)$ gives the inverse.
\end{proof}

Note that we have an inequality similar to \eqref{eqn:bullet_distance} for $d_{X \times \bR_t}$.

\begin{lemma}\label{lemma:zero_dist_on_R>0}
	Let $F,G\in \SD_{\{\tau \ge 0\}}(\bfk_{\bR_t})$.
	If $d_{\bR_t}(F,G)=0$, then $\rMS(F)=\rMS(G) $. 
\end{lemma}

\begin{proof}
	We prove $\rMS(F)\subset\rMS(G) $ by contradiction.
	Choose $(t_0;1)\in \rMS(F)\setminus \rMS (G)$. 
	There exists a neighborhood $(a,b)$ of $t_0$ in $\bR_t$ such that $(a,b)\cap \pi(\rMS(G))=\emptyset$. 
	There also exist $t_1,t_2 \in \bR$ such that $a<t_1<t_2<b$ and the restriction map $\RG((-\infty,t_2);F)\to \RG((-\infty,t_1);F)$ is not an isomorphism. 
	We may choose $t_1$ and $t_2$ arbitrarily close to $t_0$. 
	On the other hand, $\RG((-\infty,t'');G)\to \RG((-\infty,t');G)$ is an isomorphism for any $a<t'<t''<b$ by the microlocal Morse lemma. 
	Thus, an interleaving between $(\RG((-\infty,t);F))_{t\in \bR}$ and $(\RG((-\infty,t);G))_{t \in \bR}$ leads to a contradiction. 
\end{proof}

\begin{lemma}\label{lemma:zero_dist_on_R}
	Let $F,G\in \SD(\bfk_{\bR_t})$ and assume that $d_{\bR_t}(F,G)=0$.
	Then $\MS(F)=\MS(G)$ and hence  $\Supp(F)=\Supp(G)$. 
\end{lemma}

\begin{proof}
	Since $d_{\bR_t}(F\star \bfk_{[0,\infty)},G\star \bfk_{[0,\infty)})=0$, we have 
	\begin{equation}
		\MS(F) \cap \{\tau>0 \}=\rMS(F\star \bfk_{[0,\infty)})=\rMS(G\star \bfk_{[0,\infty)})=\MS(G)\cap \{\tau>0 \}
	\end{equation}
	by \cref{lemma:zero_dist_on_R>0}. 
	Similarly, we obtain 
	\begin{equation}
		\MS(F) \cap \{\tau<0 \}=\rMS(F\star \bfk_{(0,\infty)})=\rMS(G\star \bfk_{(0,\infty)})=\MS(G)\cap \{\tau<0 \}
	\end{equation}
	and hence $\rMS(F)=\rMS(G)$. 
	Note that $U\coloneqq \bR_t \setminus \pi(\rMS(F))=\bR_t \setminus \pi(\rMS(G))$ is an open subset of $\bR_t$, and both $F|_U$ and $G|_U$ are locally constant. 
	Thus $d_{X \times \bR_t}(F,G)=0$ implies $F|_U \simeq G|_U$. 
\end{proof}

\begin{lemma}\label{lemma:zero_dist_supp}
	For $F, G\in \SD(\bfk_{X \times \bR_t})$, $d_{X \times \bR_t}(F,G)=0$ implies $\Supp(F)=\Supp(G)$. 
\end{lemma}
\begin{proof}
	Let $(x,t) \in X \times \bR_t \setminus \Supp(G)$.
	Take an open neighborhood $x \in U$ in $X$ and $\varepsilon>0$ so that $U \times (t-\varepsilon,t+\varepsilon)\subset X \times \bR_t \setminus \Supp(G)$. 
	Note that for any $y \in X$, $0 \le d_{\bR_t}(F|_{\{y\} \times \bR_t},G|_{\{y\} \times \bR_t})\le d_{X \times \bR_t}(F,G)=0$ and 
	$\Supp (G|_{\{y\} \times \bR_t}) \subset \Supp(G) \cap (\{y\} \times \bR_t)$. 
	By \cref{lemma:zero_dist_on_R}, $F|_{\{y\}\times (t-\varepsilon,t+\varepsilon)}\simeq 0$ for any $y\in U$. 
	Hence, we obtain $F|_{U\times (t-\varepsilon,t+\varepsilon)}\simeq 0$ and $\Supp(F)\subset \Supp(G)$. 
	We obtain the converse inclusion $\Supp(F) \supset \Supp(G)$ similarly. 
\end{proof}

\begin{lemma}\label{lemma:zero_distance}
	Let $F, G \in \SD(\bfk_{X \times \bR_t})$ and assume that $d_{X \times \bR_t}(F,G)=0$ and $\Supp(G)\subset X \times \{0\}$.
	Then $F\simeq G$. 
\end{lemma}

\begin{proof}
	By \cref{lemma:zero_dist_supp}, we have $\Supp(F)=\Supp(G) \subset X \times \{0\}$. 
	We may write $F \simeq F'\boxtimes \bfk_{\{0\}}$ and $G \simeq G'\boxtimes \bfk_{\{0\}}$ with some $F',G'\in \SD(\bfk_X)$. 
	Take $\varepsilon >0$ arbitrarily. 
	Note that $\frakK_{\varepsilon}\circ F\simeq F'\boxtimes \bfk_{[-\varepsilon,\varepsilon]}$ and $\frakK_{\varepsilon}\circ G\simeq G'\boxtimes \bfk_{[-\varepsilon,\varepsilon]}$. 
	Hence, there exist morphisms $\alpha\colon F'\boxtimes \bfk_{[-\varepsilon,\varepsilon]}\to G'\boxtimes \bfk_{\{0\}}$ and $\beta\colon G'\boxtimes \bfk_{[-\varepsilon,\varepsilon]}\to F'\boxtimes \bfk_{\{0\}}$ such that $(\alpha,\beta)$ is a $\varepsilon$-isomorphism of $(F,G)$. 
	Restricting $\alpha$ and $\beta$ on $X \times \{0\}$, we obtain isomorphisms between $F'$ and $G'$. 
\end{proof}

Note that Petit--Schapira--Waas~\cite{petit2021property} proved that $\dist(F,G)=0$ if and only if $F \simeq G$ when $F$ and $G$ are constructible sheaves up to infinity on a real analytic manifold.
This result, as well as its proof, is different from ours and is not related to the Tamarkin category.

\cref{prop:SQ_limit_derived} shows that we can associate a sheaf with an element of the metric completion of $\Ham_c(T^*M,\omega)$ with respect to the Hofer metric $d_H$ as follows.
Let $(\varphi_n)_{n \in \bN}$ be a Cauchy sequence of Hamiltonian diffeomorphisms and $K_n \coloneqq K^{\varphi_n}$ for $n \in \bN$. 
By \cref{prop:SQ_limit_derived}, we obtain a limit object $K_\infty$ of the sequence $(K_n)_{n \in \bN}$.
The argument in the proof of the proposition also shows that another Cauchy sequence equivalent to $(\varphi_n)_{n \in \bN}$ gives the same limit object $K_\infty$ up to isomorphism. 

\begin{definition}\label{def:SQ_completion}
	Let $[(\varphi_n)_{n \in \bN}]$ be an element of the completion of $\Ham_c(T^*M,\omega)$ with respect to $d_H$. 
	The limit sheaf $K_\infty$ defined as above is denoted by $K^{[(\varphi_n)_{n}]}$ and called the sheaf quantization of $[(\varphi_n)_{n \in \bN}]$.
\end{definition}

We use the above construction to obtain a sheaf quantization of a Hamiltonian homeomorphism of $T^*M$. 

\begin{definition}[{Oh--M\"uller~\cite{oh2007group}}]\label{definition:hameo}
	Let $\phi=(\phi_s)_s \colon T^*M \times I \to T^*M$ be an isotopy of homeomorphisms of $T^*M$. 
	The isotopy $\phi$ is said to be a \emph{continuous Hamiltonian isotopy} if there exist a compact subset $C \subset T^*M$ and a sequence of smooth functions $H_n \colon T^*M \times I \to \bR$ timewisely supported in $C$ satisfying the following two conditions.
	\begin{itemize}
		\item[(1)] The sequence of flows $(\phi^{H_n})_{n \in \bN}$ $C^0$-converges to $\phi$, uniformly in $s \in I$.
		\item[(2)] The sequence $(H_n)_{n \in \bN}$ converges uniformly to a continuous function $H \colon T^*M \times I \to \bR$. 
		That is, $\|H_n-H\|_\infty \to 0$.
	\end{itemize}
	In this case, $H$ is said to generate $\phi$. 
	A homeomorphism of $T^*M$ is called a \emph{Hamiltonian homeomorphism} if it is the time-1 map of a continuous Hamiltonian isotopy.
\end{definition}

The following uniqueness theorems hold.
\begin{enumerate}
	\item A continuous Hamiltonian function generates a unique continuous Hamiltonian isotopy (Oh--M\"uller~\cite{oh2007group}).
	\item A continuous Hamiltonian isotopy can be generated by a unique continuous Hamiltonian function up to addition of a function of time (Viterbo~\cite{viterbo2006uniqueness}).
\end{enumerate}

A continuous Hamiltonian isotopy $\phi$ defines an element of the metric completion of $\Ham_c(T^*M,\omega)$ with respect to $d_H$. 
Indeed, for a sequence $(H_n)_{n \in \bN}$ satisfying the condition~(2) in \cref{definition:hameo}, $(\phi^{H_n}_1)_{n \in \bN}$ forms a Cauchy sequence with respect to the Hofer metric $d_H$. 
Moreover, the element in the metric completion is independent of the choice of a sequence $(H_n)_{n \in \bN}$.
Thus we obtain a sheaf $K_\infty$ as in \cref{def:SQ_completion}.
We give a bound of the microsupport of $K_\infty$. 
Set $\varphi_n=\phi^{H_n}_1$ and $K_n \coloneqq K^{\varphi_n}$. 
By construction, after taking a subsequence of $(K_n)_{n \in \bN}$ if necessary, $K_\infty \simeq \hocolim_n \bfk_{M \times [-a_{\ge n}, a_{\ge n}]} \star K_n$, where $(a_{\ge n})_{n \in \bN}$ is as in \cref{theorem:limit_thickening}.
Hence, we have  
\begin{equation}\label{equation:ms_SQ_ham_homeo}
	\begin{aligned}
		\mathring{\MS}(K_\infty)
		\subset
		\left\{ ((x';\xi'), (x;-\xi), (t;\tau)) 
		\; \middle| \; 
		\begin{aligned}
			& (x;\xi) \in T^*M, \\
			& (x';\xi'/\tau)=\varphi_\infty(x;\xi/\tau)    
		\end{aligned}
		\right\}
	\end{aligned}
\end{equation}
by the condition~(1) and \cref{lemma:ms_hocolim}.

\begin{remark}
	For the microsupport estimate above, we do not need the full convergence of flows, but only the convergence of the time-1 maps, that is, $\varphi_n \to \varphi_\infty \ (n \to \infty)$.
\end{remark}

\begin{definition}\label{definition:SQ_ham_homeo}
	Let $\phi$ be a continuous Hamiltonian isotopy associated with a continuous function $H \colon T^*M \times I \to \bR$ and $\varphi_\infty \coloneqq \phi_1$ a Hamiltonian homeomorphism. 
	The limit sheaf $K_\infty$ defined as above is denoted by $K^{H}_1=K^{\varphi_\infty}$ and called the sheaf quantization of the Hamiltonian homeomorphism $\varphi_\infty$. 
	One also sets $\cK^{\varphi_\infty} \coloneqq P_l(K^{\varphi_\infty}) \in \cD(M^2)$.
\end{definition}

We can also prove the following, which justifies the notation $K^{\varphi_\infty}$.

\begin{proposition}
	In the situation of \cref{definition:SQ_ham_homeo}, the object $K^H_1$ depends only on the time-1 map $\phi_1$.
\end{proposition}

\begin{proof}
	It suffices to show that if $\phi_1 = \id_{T^*M}$ then $K^H_1 \simeq \bfk_{\Delta_M \times \{0\}}$.
	We first prove the following lemma.
	
	\begin{lemma}\label{proposition:small_interval_id}
		Let $G \in \SD(\bfk_{\bR_t})$ with $d_{\bR_t}(G,\bfk_{\{0\}})<\infty$.
		Assume that there exists $\varepsilon>0$ such that for any $a<b$ with $b-a <\varepsilon$, one has $G \star \bfk_{[a,b)} \simeq \bfk_{[a,b)}$ and $G \star \bfk_{(a,b]} \simeq \bfk_{(a,b]}$.
		Then, $G \simeq \bfk_{\{0\}}$.
	\end{lemma}
	
	\begin{proof}
		First, we prove $\mathring{\MS}(G \star \bfk_{[0,\infty)})\subset \{(0;\tau)\mid \tau>0\}$ by contradiction. 
		Note that we already know that $\MS(G \star \bfk_{[0,\infty)}) \subset \{ \tau \ge 0 \}$ by \cref{lem:conv-hom-MS}. 
		Assume that there exists $(t_0;1)\in \mathring{\MS}(G \star \bfk_{[0,\infty)})$ with $t_0\neq 0$. 
		Then, there exist $t_1,t_2\in \bR$ such that $t_1<t_2$ and the restriction map $\RG ((-\infty,t_2);G \star \bfk_{[0,\infty)})\to \RG ((-\infty,t_1);G \star \bfk_{[0,\infty)})$ is not an isomorphism. 
		We may choose $t_1$ and $t_2$ arbitrary close to $t_0$ and hence may assume $t_2-t_1<\min\{t_1, \varepsilon\}$ if $t_0>0$ and $t_2<0, t_2-t_1<\varepsilon$ if $t_0<0$.
		Applying $G \star (\ast)$ to the exact triangle $\bfk_{[0,t_2-t_1)} \to \bfk_{[0,\infty)} \to \bfk_{[t_2-t_1, \infty)} \toone $, we have an exact triangle 
		\begin{equation}
			\bfk_{[0,t_2-t_1)} \to G \star \bfk_{[0,\infty)} \to T_{t_2-t_1} G \star \bfk_{[0, \infty)} \toone.
		\end{equation}
		Since $\RG ((-\infty, t_2);\bfk_{[0,t_2-t_1)})=0$, we get an isomorphism 
		\begin{equation}
			\RG ((-\infty,t_2);G \star \bfk_{[0,\infty)}) \simto \RG ((-\infty,t_1);G \star \bfk_{[0,\infty)}).
		\end{equation}
		The morphism coincides with the restriction map by the construction, which is a contradiction. 
		
		Similarly, we obtain $\mathring{\MS}(G \star \bfk_{(0,\infty)})\subset \{(0;\tau)\mid \tau<0\}$. 
		
		We get $\mathring{\MS}(G) \subset T^*_0 \bR_t$ by applying $G \star (\ast)$ to the exact triangle $\bfk_{(0,\infty)} \to \bfk_{[0,\infty)} \to \bfk_{\{0\}} \toone$. 
		Hence, we obtain the desired isomorphism from $d_{\bR_t}(G,\bfk_{\{0\}})<\infty$.
	\end{proof}
	
	Now let $\phi$ be a continuous Hamiltonian isotopy generated by a continuous function $H \colon T^*M \times I \to \bR$ and assume that $\phi_1 = \id_{T^*M}$. 
	By the microsupport estimate~\eqref{equation:ms_SQ_ham_homeo}, we get $\pi(\mathring{\MS}(K^H_1)) \subset \Delta_M \times \bR$. 
	By construction, $d_{M^2 \times \bR_t}(\bfk_{\Delta_M \times \{0\}},K^H_1)<\infty$.
	Restricting to $(M^2 \setminus \Delta_M) \times \bR_t$, we find that $d_{(M^2 \setminus \Delta_M) \times \bR_t}(0,K^H_1|_{(M^2 \setminus \Delta_M) \times \bR_t})<\infty$, which implies $K^H_1|_{(M^2 \setminus \Delta_M) \times \bR_t} \simeq 0$.
	Hence, we can write $K^H_1 \simeq (\delta_M \times \id_{\bR_t})_* K'$ with $K' \in \SD(\bfk_{M \times \bR_t})$.
	Again by the estimate~\eqref{equation:ms_SQ_ham_homeo} and \cref{proposition:SSpushpull}(i), we find that $\MS(K') \subset 0_M \times T^*\bR_t$. 
	
	Let us keep the notation of a compact set $C \subset T^*M$ and $K_n=K^{\varphi_n}$ as above (see also \cref{definition:hameo}).
	For $F \in \SD(\bfk_{M\times \bR_t})$ such that 
	\begin{equation}
		\mathring{\MS}(F) \cap \{(x,t;\xi,\tau) \mid \tau\neq 0, (x;\xi/\tau)\in C \}=\emptyset,
	\end{equation}
	we have $K_n\bullet F \simeq F$ for any $n \in \bN$, which implies $d_{M\times \bR_t}(K^H_1 \bullet F, F)=0$. 
	For each $x_0 \in M$ and $a<b \in \bR$ with sufficiently small $b-a$, there exists $F \in \SD(\bfk_{M\times \bR_t})$ such that (1)~its microsupport does not intersect the cone $\{(x,t;\xi,\tau) \mid \tau\neq 0, (x;\xi/\tau)\in C \}$ of $C$ and (2)~$F|_{\{x_0\} \times \bR_t}\simeq \bfk_{[a,b)}$. 
	For example, such $F$ is obtained as the sheaf quantization of an exact Lagrangian immersion of a sphere.  
	Therefore, $d_{\bR_t}(K'|_{\{x_0\}\times\bR_t}\star \bfk_{[a,b)}, \bfk_{[a,b)})\leq d_{M\times\bR_t}(K^H_1 \bullet F, F)=0$ if $b-a$ is sufficiently small. 
	Applying \cref{lemma:zero_dist_on_R} to $K'|_{\{x_0\}\times\bR_t}\star \bfk_{[a,b)}$ and $\bfk_{[a,b)}$, we get $K'|_{\{x_0\}\times\bR_t}\star \bfk_{[a,b)}\simeq \bfk_{[a,b)}$. 
	Similarly, we obtain $K'|_{\{x_0\}\times\bR_t}\star \bfk_{(a,b]}\simeq \bfk_{(a,b]}$ for any $x_0 \in M$ and sufficiently small $b-a$. 
	
	Hence, for any $x_0 \in M$, $K'|_{\{x_0\} \times \bR_t} \in \SD(\bfk_{\bR_t})$ satisfies the condition of \cref{proposition:small_interval_id}.
	Combining this with the estimate of $\MS(K')$, we may write $K'=L \boxtimes \bfk_{\{0\}}$, where $L$ is a rank one local system.
	Again by the fact that $d_{M^2 \times \bR_t}(\bfk_{\Delta_M \times \{0\}},K^H_1)<\infty$, we conclude that $L=\bfk_M$, which proves $K^H_1 \simeq \bfk_{\Delta_M \times \{0\}}$.
\end{proof}

\section{Spectral invariants in Tamarkin category}\label{section:spectral}

In this section, we define spectral invariants for an object of the Tamarkin category and develop Lusternik--Schnirelmann theory.
Most of the definitions and the results were announced to appear in Humili\`ere--Vichery~\cite{HV}.

\begin{definition}
	Let $F \in \cD(M)$.
	\begin{enumerate}
		\item For $n \in \bZ$, one defines
		\begin{align}
			Q_c^n(F) 
			& \coloneqq 
			\Hom(\bfk_{M \times [0,\infty)}, {T_c} F[n]) \\
			& \simeq 
			H^n \RHom(\bfk_{M \times [-c,\infty)}, F), \notag \\
			Q_\infty^n(F) & \coloneqq \varinjlim_{c \to \infty} Q_c^n(F).
		\end{align}
		Set $Q_c^*(F) \coloneqq \bigoplus_{n \in \bZ} Q_c^n(F)$ and similarly for $Q_\infty^*(F)$.
		Denote the canonical map by $i_c \colon Q_c^*(F) \to Q_\infty^*(F)$. 
		\item For $\alpha \in Q_\infty^*(F)$, one defines 
		\begin{equation}
			c(\alpha;F) \coloneqq \inf\{ c \in \bR \mid \alpha \in \Image i_c \} 
		\end{equation}
		and calls it the \emph{spectral invariant} of $F$ for $\alpha$.
		\item One defines 
		\begin{equation}
			\Spec(F) \coloneqq \{ c(\alpha;F) \in \bR \mid \alpha \in Q_\infty^*(F) \} \subset \bR.
		\end{equation}
	\end{enumerate}
\end{definition}

\begin{remark}
	In our definition, $c(0;F)=-\infty$. 
	In general, $c(\alpha; F)$ can be $-\infty$ for non-zero $\alpha \in Q_{\infty}^*(F)$. 
	We give such an example when $M=\mathrm{pt}$.
	Let $G \coloneqq \prod_{n \in \bZ_{\ge 1}} \bfk_{[-n,n)}[1]$ and define $g \colon \bfk_{[0,\infty)} \to G$ to be the product $\prod_{n \in \bZ_{\ge 1}} g_n$, where $g_n \colon \bfk_{[0,\infty)} \to \bfk_{[-n,n)}[1]$ is a non-trivial morphism.
    Set $F \coloneqq P_l(G)$, the projection to $\cD(\pt)$. 
    Then the projector $P_l$ induces a morphism $\tilde{g} \colon \bfk_{[0,\infty)} \to F$ in $\cD(\pt)$, which satisfies $[\tilde{g}] \neq 0$ in $Q_\infty^0(F)$ and $c([\tilde{g}]; F)=-\infty$.
\end{remark}

Let $t \colon M \times \bR_t \to \bR_t$ be the projection and let $\Gamma_{dt}$ denote the graph of the $1$-form $dt$:
\begin{equation}
	\Gamma_{dt} \coloneqq \{ (x,t;0,1) \mid (x,t) \in M \times \bR_t \}.
\end{equation}
Note that 
\begin{equation}
	\Spec(F) \subset \{ -c \in \bR_t \mid c \in t\pi(\MS(F) \cap \Gamma_{dt}) \}.
\end{equation}

In order to state Lusternik--Schnirelmann theory for sheaves, we recall the algebraic counterpart of cup-length, which is studied in \cite{AISQ}.

\begin{definition}
	Let $R$ be an associative (not necessarily commutative) non-unital ring over $\bfk$. 
	For a right $R$-module $A$, one defines 
	\begin{equation}
		\cl_R(A) 
		\coloneqq 
		\inf 
		\lc 
		k-1 \; \middle| \; 
		\begin{aligned}
			& k \in \bN, a_0 \cdot r_1 \cdots r_k =0 \\
			& \text{for any $a_0 \in A$ and $(r_1,\dots , r_k) \in R^k$}
		\end{aligned}
		\rc
		\in \bZ_{\ge -1} \cup \{\infty\}.
	\end{equation}
\end{definition}

We note that $\cl_R(A)=-1$ if and only if $A=0$. 
If there is no risk of confusion, we simply write $\cl(A)$ for $\cl_R(A)$.
By definition, we have the following lemma.

\begin{lemma}\label{lemma:cl_exact}
	For an exact sequence of right $R$-modules $0 \to A \to B \to C \to 0$, one has 
	\begin{equation}
		\cl(B) \le \cl(A) + \cl(C) +1.
	\end{equation}
\end{lemma}

Let $F \in \cD(M)$.
Then, we have a right action of $\End(\bfk_{M \times [0,\infty)}) \simeq H^*(M;\bfk)$ on $Q_c^*(F)$, which induces an action on $Q_\infty^*(F)$. 
Hereafter we set $R \coloneqq \bigoplus_{n \ge 1} H^n(M;\bfk)$ and consider the cup-length over $R$.
Note that the cup-length over $R$ is always finite. 

\begin{theorem}\label{theorem:sheaf_spectral_invariant}
	Let $F \in \cD(M)$. 
	Assume that $t$ is proper on $\Supp(F)$ and there exists $c \ll 0$ satisfying $i_c=0$.
	Let $\pi_M \colon T^*(M \times \bR_t) \to M$ denote the projection.
	If $\# \Spec(F) \le \cl(Q_\infty^*(F))$, then there exists $c \in \Spec(F)$ such that $\pi_M(\MS(F) \cap \Gamma_{dt} \cap \pi^{-1} t^{-1}(-c))$ is cohomologically non-trivial in $M$. 
	That is, for any open neighborhood $U$ of $\pi_M(\MS(F) \cap \Gamma_{dt} \cap \pi^{-1} t^{-1}(-c))$, the restriction map $\bigoplus_{n \ge 1} H^n(M;\bfk) \to \bigoplus_{n \ge 1} H^n(U;\bfk)$ is non-zero.
\end{theorem}

For $F \in \cD(M)$, if $F|_{M \times (c,\infty)}$ is locally constant for $c \gg 0$, then $i_c=0$ for $c \ll 0$.
If the conclusion holds, then $\pi_M(\MS(F) \cap \Gamma_{dt})$ is also cohomologically non-trivial in $M$.

For the proof of the theorem, we prepare some notation. 
For $d \in \bR$, we define 
\begin{equation}
	Q_{\infty, d}^*(F) \coloneqq \Image(i_d \colon Q_d^*(F) \to Q_\infty^*(F)) \subset Q_\infty^*(F).
\end{equation}
Then we get the following properties:
\begin{enumerate}
	\item[(1)] If $d<d'$, then $Q_{\infty, d}^*(F) \subset Q_{\infty, d'}^*(F)$. 
	\item[(2)] If $[d,d'] \cap \Spec(F) = \emptyset$, then $Q_{\infty, d}^*(F) \simeq Q_{\infty, d'}^*(F)$. 
	\item[(3)] For $d<d'$, there exists an exact sequence of right $H^*(M;\bfk)$-modules
	\begin{equation}\label{eq:Qinfty_exact}
		0 \to Q_{\infty, d}^*(F) \to Q_{\infty, d'}^*(F) \to Q_{\infty, d'}^*(F)/Q_{\infty, d}^*(F) \to 0.
	\end{equation}
	Moreover, we have 
	\begin{equation}
		\cl(H^*_{M \times [-d',-d)}(M \times \bR;F)) \ge \cl(Q_{\infty, d'}^*(F)/Q_{\infty, d}^*(F)).
	\end{equation}
\end{enumerate}

\begin{proof}[Proof of \cref{theorem:sheaf_spectral_invariant}]
	If $Q_\infty^*(F) \simeq 0$, then $\cl(Q_\infty^*(F))=-1$ and there is nothing to prove.
	
	Suppose that $Q_\infty^*(F) \neq 0$.
	Since $\cl(Q_\infty^*(F))$ is finite, we may assume that $\Spec(F)$ is finite and set $\Spec(F)=\{c_1, \dots, c_N \}$ with $c_1 < c_2 < \dots < c_N$.
	Let $d_0,d_1, \dots, d_N \in \bR$ such that $d_0 < c_1 < d_1< \dots < d_{N-1} < c_N < d_N$.
	Note that $Q_{\infty, d_0}^*(F)=0$ by the assumption and $Q_{\infty, d_N}^*(F)=Q_\infty^*(F)$. 
	Applying \cref{lemma:cl_exact} to the exact sequence \eqref{eq:Qinfty_exact} with $d=d_{i-1}, \allowbreak d'=d_i$, by induction we get 
	\begin{equation}
		\cl(Q_\infty^*(F)) \le N-1 + \sum_{i=1}^N \cl(Q_{\infty, d_i}^*(F)/Q_{\infty, d_{i-1}}^*(F)).
	\end{equation}
	Hence if $\# \Spec(F)=N \le \cl(Q_\infty^*(F))$, there exists $i \in \{1,\dots, N\}$ such that 
	\begin{equation}
		\cl(Q_{\infty, d_i}^*(F)/Q_{\infty, d_{i-1}}^*(F)) \ge 1.    
	\end{equation}
	
	For such $i$ above, we claim that $\pi_M(\MS(F) \cap \Gamma_{dt} \cap \pi^{-1} t^{-1}(-c_i))$ is cohomologically non-trivial in $M$.
	For $c \in \bR$ and $I \subset \bR$, set 
	\begin{equation}
		K_{c} \coloneqq \pi_M(\MS(F) \cap \Gamma_{dt} \cap \pi^{-1} t^{-1}(-c)), \
		K_I \coloneqq \bigcup_{c \in I} K_c \ \subset M.
	\end{equation}
	Let $U$ be any open neighborhood of $K_{c_i}$ in $M$. 
	We take $K_{c_i} \subset U_0 \Subset U_1 \Subset U$ and a $C^\infty$-function $\rho \colon M \to \bR$ such that 
	\begin{enumerate}
		\item[(1)] $\rho(x) \in [-1,1]$ for any $x \in M$,
		\item[(2)] $\rho|_{U_0} \equiv 1$, 
		\item[(3)] $\rho|_{M \setminus U_1} \equiv -1$. 
	\end{enumerate}
	Then there exists $\varepsilon>0$ such that $K_{[c_i-\varepsilon, c_i+\varepsilon]} \subset U_0$. 
	Taking $0 < \varepsilon' \ll \varepsilon$ and setting $\rho' \coloneqq \varepsilon' \rho \colon M \to \bR$, we may assume 
	\begin{equation}
		(x,t;s\rho'(x),1) \not\in \MS(F)
	\end{equation}
	for $x \in U_1 \setminus U_0, t \in [-c_i-\varepsilon',-c_i+\varepsilon']$, and $s \in [0,1]$.
	Hence, we can apply the microlocal Morse lemma to $F$ and $(V_s)_{s \in [0,1]}$, where 
	\begin{equation}
		V_s \coloneqq \{ (x,t) \in M \times \bR_t \mid t < -c_i+(s-1)\varepsilon' -s \rho'(x) \},
	\end{equation}
	and obtain an isomorphism 
	\begin{equation}
		\RG(M \times (-\infty, -c_i-\varepsilon');F)=\RG(U_0;F)
		\simto \RG(U_1;F).
	\end{equation}
	We set $X=M \times \bR$ and 
	\begin{equation}
		Z \coloneqq
		M \times (-\infty, -c_i+\varepsilon') \setminus U_1
		=
		\{ (x,t) \in M \times \bR_t \mid -c_i-\rho'(x) \le t < -c_i+\varepsilon'\}.
	\end{equation}
	By the above isomorphism, we get a morphism of exact triangles
	\begin{equation}
		\xymatrix@C=8pt{
			\RG_{M \times [-c_i+\varepsilon',\infty)}(X;F) \ar[r] \ar[d]^-{\rotatebox{90}{$\sim$}} & \RG_{M \times [-c_i+\varepsilon', \infty) \cup Z}(X;F) \ar[r] \ar[d]^-{\rotatebox{90}{$\sim$}} & \RG_{Z}(X;F) \ar[r] \ar[d] &\\
			\RG_{M \times [-c_i+\varepsilon',\infty)}(X;F) \ar[r] & \RG_{M \times [-c_i-\varepsilon',\infty)}(X;F) \ar[r] & \RG_{M \times [-c_i-\varepsilon',-c_i+\varepsilon')}(X;F) \ar[r] & ,
		}
	\end{equation}
	where the middle vertical morphism is an isomorphism by the above argument.
	Hence, by the five lemma, we have an isomorphism 
	\begin{equation}
		\RG_{M \times [-c_i-\varepsilon', -c_i+\varepsilon')}(M \times \bR_t;F) \xleftarrow{\sim} \RG_Z(M \times \bR_t;F) 
		\simeq \RHom(\bfk_Z,F).
	\end{equation}
	Since $\Supp(\bfk_Z) =\overline{Z} \subset \overline{U_1} \times [-c_i-\varepsilon', -c_i+\varepsilon']$, the action of $H^*(M;\bfk)$ on $H^*_Z(M \times \bR_t;F)$ factors through $H^*(\overline{U_1})$.
	Hence, we have
	\begin{equation}
		\begin{aligned}
			\cl(H^*(\overline{U_1})) 
			& \ge \cl(H^*_Z(M \times \bR_t;F)) \\
			& = \cl(H^*_{M \times [-c_i-\varepsilon', -c_i+\varepsilon')}(M \times \bR_t;F)) \\
			& \ge \cl(Q_{\infty, c+\varepsilon'}^*(F)/Q_{\infty, c-\varepsilon'}^*(F)) \ge 1.
		\end{aligned}
	\end{equation}
	Thus, we conclude that $\overline{U_1}$ is cohomologically non-trivial in $M$, which implies that $U$ is also cohomologically non-trivial in $M$.
\end{proof}

We consider the spectral invariants for the sheaf associated with a Hamiltonian diffeomorphism/homeomorphism and a compact exact Lagrangian submanifold.

Let $L$ be a compact exact Lagrangian submanifold of $T^*M$. 
Take a function $f \colon L \to \bR$ satisfying $\alpha_{T^*M}|_{L} = df$ and define 
\begin{equation}
	\widehat{L} \coloneqq \{ (x,t;\xi,\tau) \mid \tau >0, (x;\xi/\tau) \in L, t=-f(x;\xi/\tau) \}.
\end{equation}
In this setting, Guillermou~\cite{Gu12} (see also \cite{Gu23, Viterbo-Sheaves}) proved the existence and the uniqueness of an object $F_L \in \cD(M)$ that satisfies $\rMS(F_L) =\widehat{L}$ and $F_L|_{M \times (c,\infty)} \simeq \bfk_{M \times (c,\infty)}$ for a sufficiently large $c>0$.
We call $F_L$ the canonical simple sheaf quantization of $L$. 
When $L=0_M$ and $f \equiv 0$, we have $F_{0_M} \simeq \bfk_{M \times [0,\infty)}$.

Moreover, let $\varphi \in \Ham_c(T^*M,\omega)$ be a compactly supported Hamiltonian diffeomorphism.
We define the set of spectral invariants of $\Spec(\varphi,L)$ of the Lagrangian submanifold $\varphi(L)$ by 
\begin{equation}
	\Spec(\varphi,L) \coloneqq \Spec(\cK^\varphi \bullet F_L), 
\end{equation}
where $\cK^\varphi \in \cD(M^2)$ is the sheaf quantization of $\varphi$.
This set is well-defined up to shift. 
By a result of Viterbo~\cite{Viterbo-Sheaves}, $\Spec(\varphi,L)$ is equal to the set of the Floer-theoretic spectral invariants associated with $0_M$ and $\varphi(L)$.
If two Hamiltonian diffeomorphisms $\varphi, \varphi' \in \Ham_c(T^*M,\omega)$ satisfy $\varphi(L)=\varphi'(L)$, then there exists some constant $C \in \bR$ such that
\begin{equation}
	\Spec(\cK^\varphi \bullet F_L) = \Spec(\cK^{\varphi'} \bullet F_L) + C.
\end{equation}
Indeed, since both of $\cK^\varphi \bullet F_L$ and $\cK^{\varphi'} \bullet F_L$ are canonical simple sheaf quantizations of $\varphi(L)$, by the uniqueness result, we have $\cK^\varphi \bullet F_L \simeq T_{c} (\cK^{\varphi'} \bullet F_L)$ for some $c \in \bR$.

Let $\varphi_\infty$ be a Hamiltonian homeomorphism and $(H_n)_{n \in \bN}$ a sequence of smooth functions that satisfies the condition in \cref{definition:hameo}.
For any $n \in \bN$, we set $\varphi_n \coloneqq \phi^{H_n}_1$ and consider the sheaf quantization $K^{\varphi_n}$ of $\varphi_n$. 
Then, the sequence $(\cK^{\varphi_n})_{n \in \bN}$ forms a Cauchy sequence with respect to $d_{\cD(M^2)}$ and gives an object  $\cK^{\varphi_\infty}$. 
In this situation, we define the set of spectral invariants $\Spec(\varphi_\infty, L)$ by 
\begin{equation}
	\Spec(\varphi_\infty, L) \coloneqq \Spec(K^{\varphi_\infty} \bullet F_{L}).
\end{equation}
This is well-defined up to shift.

\begin{lemma}\label{lemma:spec_hameo}
	One has 
	\begin{equation}
		\Spec(\varphi_\infty, L) = \lim_{n \to \infty} \Spec(\varphi_n, L).
	\end{equation}
\end{lemma}

\begin{proof}
	Since $d_{\cD(M^2)}(\cK^{\varphi_n},\cK^{\varphi_\infty}) \to 0 \ (n \to \infty)$, we have $d_{\cD(M)}(\cK^{\varphi_n} \bullet F_L,\cK^{\varphi_\infty} \bullet F_L) \to 0 \ (n \to \infty)$.
	Hence, we get the result.
\end{proof}

The spectral norm for Lagrangian submanifolds is defined as follows.

\begin{definition}\label{definition:spectral_norm}
	Let $\varphi \colon T^*M \to T^*M$ be a Hamiltonian diffeomorphism.
	One defines
	\begin{equation}
		\gamma (\varphi(0_M))\coloneqq \max\Spec(\cK^\varphi \bullet \bfk_{M \times [0,\infty)})+\max\Spec(\cK^{\varphi^{-1}} \bullet \bfk_{M \times [0,\infty)})
	\end{equation}
	and calls it the \emph{spectral norm} of $\varphi(0_M)$.
\end{definition}

Note that the spectral norm $\gamma(\varphi(0_M))$ depends only on the image $\varphi(0_M)$. 
This follows from \cref{proposition:spectral_norm_distance} below and the fact that if $\varphi, \varphi' \in \Ham_c(T^*M,\omega)$ satisfy $\varphi(0_M)=\varphi'(0_M)$ then $\cK^\varphi \bullet \bfk_{M \times [0,\infty)} \simeq T_c(\cK^{\varphi'} \bullet \bfk_{M \times [0,\infty)})$ for some constant $c \in \bR$. 
We also remark that $\gamma(\varphi(0_M))$ in \cref{definition:spectral_norm} is the same as that in \cite{viterbo1992symplectic,oh97,oh99} by the above argument.

We describe this spectral norm in terms of the distance on the Tamarkin category. 
For that purpose, we introduce an interleaving distance up to shift.

\begin{definition}
	For $F,G \in \cD(M)$, one defines 
	\begin{equation}
		\overline{d_{\cD(M)}}(F,G)\coloneqq \inf_{c\in\bR} d_{\cD(M)}(F,T_cG).
	\end{equation}
\end{definition}

\begin{proposition}\label{proposition:spectral_norm_distance}
	For a Hamiltonian diffeomorphism $\varphi \colon T^*M \to T^*M$, one has 
	\begin{equation}
		\gamma (\varphi(0_M))=\overline{d_{\cD(M)}}( \bfk_{M \times [0,\infty)},\cK^\varphi \bullet \bfk_{M \times [0,\infty)}).
	\end{equation}
\end{proposition}

\begin{proof}
	We will argue similarly to the proof of \cref{theorem:SQ_inequality}. 
	Again let $\Tor$ be the full triangulated subcategory of $\cD(M)$ consisting of torsion objects $\{F \mid d_{\cD(M)}(F, 0)<\infty\}$. 
	Then, the Hom set of the localized category $\cD(M)/\Tor$ is computed as 
	\begin{equation}
		\Hom_{\cD(M)/\Tor}(F,G) \simeq \varinjlim_{c\to \infty}\Hom_{\cD(M)}(F,T_cG). 
	\end{equation}
	Hence, $Q_\infty^*(F)$ is isomorphic to $\bigoplus_n \Hom_{\cD(M)/\Tor}(\bfk_{M \times [0,\infty)},F[n])$. 
	
	For an object $F\in \cD(M)$ with $d(\bfk_{M \times [0,\infty)}, F)<\infty$, $F$ and $\bfk_{M \times [0,\infty)}$ are isomorphic in $\cD(M)/\Tor$. 
	On the other hand, any isomorphism $\bar{\alpha}\in \Hom_{\cD(M)/\Tor}(\bfk_{M \times [0,\infty)},F) \simeq Q_\infty^0(F)$ gives an isomorphism $Q_\infty^*(F)\simeq \bigoplus_n \Hom_{\cD(M)/\Tor}(\bfk_{M \times [0,\infty)},\bfk_{M \times [0,\infty)}[n])\simeq H^*(M)$ of right $H^*(M)$-modules that sends $\bar{\alpha}\in Q_\infty^0(F)$ to $1\in H^0(M)\simeq \bfk$. 
	Note that 
	\begin{equation}
		Q_{\infty,c}^0(F)\simeq
		\begin{cases}
			\bfk &(c>c(\bar{\alpha},F))\\
			0 &(c<c(\bar{\alpha},F))
		\end{cases}
	\end{equation}
	by definition. 
	Since $Q_{\infty,c}^*(F)\subset Q_\infty^*(F)$ is an $H^*(M)$-submodule for any $c$, 
	$Q_{\infty,c}^0(F)\neq 0$ if and only if $Q_{\infty,c}^*(F)= Q_\infty^*(F)$. 
	Hence we obtain 
	\begin{equation}
		\max\Spec(F)=c(\bar{\alpha},F). 
	\end{equation}
	
	\noindent (i) 
	Let $a,b\in \bR$ and 
	\begin{align}
		& \alpha\colon \bfk_{M \times [0,\infty)}\to T_a \cK^\varphi \bullet\bfk_{M \times [0,\infty)}, \\
		& \beta\colon \cK^\varphi \bullet \bfk_{M \times [0,\infty)}\to T_b\bfk_{M \times [0,\infty)}
	\end{align}
	such that $\alpha $ descends to an isomorphism $\bar{\alpha}\in \Hom_{\cD(M)/\Tor}(\bfk_{M \times [0,\infty)},\cK^\varphi \bullet \bfk_{M \times [0,\infty)})$ and $\beta$ descends to its inverse $\bar{\beta}$. 
	Since $\bar{\alpha}\in Q_{\infty,a}^0(\cK^\varphi \bullet \bfk_{M \times [0,\infty)})$ is non-zero, $a\geq \max\Spec(\cK^\varphi \bullet \bfk_{M \times [0,\infty)})$. 
	On the other hand, $\beta$ gives a non-zero element of $Q_{\infty,b}^0(\cK^{\varphi^{-1}} \bullet \bfk_{M \times [0,\infty)})$ since $\cK^{\varphi^{-1}} \bullet (\ast)$ is the inverse functor of $\cK^\varphi \bullet (\ast)$. 
	Hence, we obtain $b\geq \max\Spec(\cK^{\varphi^{-1}} \bullet \bfk_{M \times [0,\infty)})$. 
	This shows $\gamma (\varphi(0_M))\leq \overline{d_{\cD(M)}}( \bfk_{M \times [0,\infty)},\cK^\varphi \bullet \bfk_{M \times [0,\infty)})$.
	
	\noindent (ii) 
	For any $a> \max\Spec(\cK^\varphi \bullet \bfk_{M \times [0,\infty)})$, there exists $\alpha\colon \bfk_{M \times [0,\infty)}\to T_a \cK^\varphi \bullet\bfk_{M \times [0,\infty)}$ that descends to an isomorphism  $\bar{\alpha}\in \Hom_{\cD(M)/\Tor}(\bfk_{M \times [0,\infty)},\cK^\varphi \bullet \bfk_{M \times [0,\infty)})$.  
	For any $b>\max\Spec(\cK^{\varphi^{-1}} \bullet \bfk_{M \times [0,\infty)})$, there also exists $\beta\colon \cK^\varphi \bullet \bfk_{M \times [0,\infty)}\to T_b\bfk_{M \times [0,\infty)}$ that descends to 
	an isomorphism $\beta$. Since $\Hom_{\cD(M)/\Tor}(\bfk_{M \times [0,\infty)},\bfk_{M \times [0,\infty)})\simeq \bfk$, we may assume that $\bar{\beta}$ is the inverse of $\bar{\alpha}$ after multiplying a non-zero element of $\bfk$ to $\beta$. 
	The composite $\bar{\beta} \circ \bar{\alpha}=\id_{\bfk_{M \times [0,\infty)}}$ can be regarded as a non-zero element of $Q_{\infty, a+b}^0(\bfk_{M \times [0,\infty)})$. 
	Noting that 
	\begin{equation}\label{equation:Qc_bfk}
		Q_{c}^0(\bfk_{M \times [0,\infty)})\simeq Q_{\infty,c}^0(\bfk_{M \times [0,\infty)})\simeq
		\begin{cases}
			\bfk &(c\ge 0)\\
			0 &(c<0),
		\end{cases}
	\end{equation}
	we obtain $a+b\ge 0$ and the preimage $T_a\beta\circ \alpha \in Q_{a+b}^0(\bfk_{M \times [0,\infty)})$ of $\id_{\bfk_{M \times [0,\infty)}}=\bar{\beta} \circ \bar{\alpha}$ is $\tau_{a+b}(\bfk_{M \times [0,\infty)})$. 
	Similarly, we obtain $T_b\alpha\circ\beta=\tau_{a+b}(\cK^{\varphi} \bullet \bfk_{M \times [0,\infty)})$ using 
	\begin{equation}
		\Hom_{\cD(M)}(\cK^{\varphi} \bullet \bfk_{M \times [0,\infty)}, T_c \cK^{\varphi} \bullet \bfk_{M \times [0,\infty)})\simeq Q_{c}^0(\bfk_{M \times [0,\infty)})
	\end{equation}
	and \eqref{equation:Qc_bfk}.
	This proves $\gamma (\varphi(0_M))\geq \overline{d_{\cD(M)}}( \bfk_{M \times [0,\infty)},\cK^\varphi \bullet \bfk_{M \times [0,\infty)})$. 
\end{proof}

\begin{remark}
	One can define Hamiltonian spectral invariants in a sheaf-theoretic way as follows. 
	Let $H \colon T^*M \times I \to \bR$ be a timewise compactly supported function and $K^H \in \SD(\bfk_{M^2 \times \bR_t \times I})$ the associated sheaf quantization.
	Then, we define 
	\begin{equation}
		\Spec(H) \coloneqq \Spec(P_l(\cHom^\star(\bfk_{\Delta_M \times [0,\infty)},K^H_1))).
	\end{equation}
	The Hamiltonian spectral norm is also defined by  
	\begin{equation}
		\gamma(H) \coloneqq \max\Spec(H)+\max\Spec(\overline{H})
	\end{equation}
	and we obtain 
	\begin{equation}
		\gamma(H)=\overline{d_{\cD(M^2)}}( \bfk_{\Delta_M \times [0,\infty)},\cK^H_1). 
	\end{equation}
	We conjecture that $\gamma(H)$ coincides with the Hamiltonian spectral norm of $H$ defined by Frauenfelder--Schlenk~\cite{FS07} for a compact manifold $M$.
\end{remark}

\section{Arnold-type principle for Hamiltonian homeomorphisms}\label{section:arnold}

In this section, we use the previous results to prove an Arnold-type theorem for a Hamiltonian homeomorphism of a cotangent bundle in a purely sheaf-theoretic way.
Throughout this section, we assume that $M$ is compact. 

Let $L$ be a compact exact Lagrangian submanifold of $T^*M$ and $F_L$ be the canonical simple sheaf quantization of $L$. 
Let $\varphi_\infty$ be a Hamiltonian homeomorphism and $K^{\varphi_\infty}$ be the sheaf quantization of $\varphi_\infty$. 
We set $F_\infty \coloneqq \cK^{\varphi_\infty} \bullet F_L$.
Then, by \eqref{equation:ms_SQ_ham_homeo} we have
\begin{align}\label{eq:ms-esitimate}
	\MS(F_\infty) \cap \{ \tau >0 \} \subset
	\{ (x,t;\xi,\tau) \mid \tau >0, (x;\xi/\tau) \in \varphi_\infty(L) \}.
\end{align}
Moreover, by construction and the property of $F_L$, we get $Q_\infty^*(F_\infty) \simeq H^*(M;\bfk)$.

Combining the previous results, we obtain the following result by a purely sheaf-theoretic method.
In the case $L$ is the zero-section $0_M$, it was proved by \cite{buhovsky2019arnold} in a more general setting (see below).
We set $\cl(M) \coloneqq \cl(H^*(M;\bfk))$, which is called the cup-length of $M$ over $\bfk$.

\begin{theorem}\label{theorem:ineq_hameo}
	Let $L$ be a compact exact Lagrangian submanifold of $T^*M$ and $\varphi_\infty$ be a Hamiltonian homeomorphism of $T^*M$.
	If $\# \Spec(\varphi_\infty, L) \le \cl(M)$, then $0_M \cap \varphi_\infty(L)$ is cohomologically non-trivial in $M$, in particular it is infinite.
\end{theorem}

\begin{proof}
	Let $t \colon M \times \bR_t \to \bR_t$ and $\pi_M \colon T^*(M \times \bR_t) \to M$ be the projections.
	Let $\Gamma_{dt}$ denote the graph of the 1-form $dt$. 
	Then, by \eqref{eq:ms-esitimate}, we have $\pi_M(\MS(F_\infty) \cap \Gamma_{dt}) \subset 0_M \cap \varphi_\infty(0_M)$.
	Thus we obtain the result by applying \cref{theorem:sheaf_spectral_invariant,lemma:spec_hameo} to $F_\infty$.
\end{proof}

By using the spectral norm $\gamma$ and its $C^0$-continuity, we can construct a sheaf quantization the image of the zero-section $0_M$ under a $C^0$-limit of Hamiltonian diffeomorphisms. 
With the sheaf quantization, we can also recover \cite[Thm.~1.1]{buhovsky2019arnold}.
Note that the proof is not purely sheaf-theoretic.

\begin{proposition}\label{proposition:arnold_hamhomeo}
	Let $\varphi_\infty \colon T^*M \to T^*M$ be a compactly supported homeomorphism.
	Assume that there exist a compact subset $C \subset T^*M$ and a sequence of Hamiltonian diffeomorphisms $(\varphi_n)_{n \in \bN} \subset \Ham_c(T^*M,\omega)$ supported in $C$ that $C^0$-converges to $\varphi_\infty$ for some Riemannian metric. 
	\begin{enumerate}
		\item There exists an object $F_\infty \in \cD(M)$ such that $d_{\cD(M)}(\bfk_{M\times [0,\infty)}, F_\infty) < \infty$ and 
		\begin{equation}
			\MS(F_\infty) \cap \{ \tau >0\} \subset \{ (x,t;\xi, t) \mid \tau > 0, (x;\xi/\tau) \in \varphi_\infty(0_M) \}.
		\end{equation}
		\item 
		There exists a sequence of real numbers $(c_n)_{n \in \bN} \subset \bR$ such that $\Spec(\varphi_\infty,0_M) \coloneqq \lim_{n \to \infty} T_{-c_n} \Spec(\varphi_n,0_M)$ is well-defined up to shift. 
		Moreover, if $\# \Spec(\varphi_\infty, 0_M) \le \cl(M)$, then $0_M \cap \varphi_\infty(0_M)$ is cohomologically non-trivial in $M$, in particular it is infinite.
	\end{enumerate}
\end{proposition}

\begin{proof}
	(i) 
	Our $\gamma (\psi (0_M))$ coincides with $\gamma (\psi (0_M))$ in \cite{buhovsky2019arnold} for any compactly supported Hamiltonian diffeomorphism $\psi$.
	By \cite[Thm.~4.1]{buhovsky2019arnold}, for any $\varepsilon >0$, there exists $\delta >0 $ such that $d_{C^0}(\psi, \id_{T^* M})<\delta$ implies $\gamma (\psi (0_M))<\varepsilon$. 
	By the $C^0$-convergence of $(\varphi_n)_{n \in \bN}$, for any $\delta >0$, there exists $N\in \bN$ such that if $n,m \ge N$, then $d_{C^0}(\varphi^{-1}_n\varphi_m, \id_{T^*M}) <\delta$. 
	Hence, for any $\varepsilon >0$, there exists $N\in \bN$ such that if $n,m \ge N$, then $\gamma (\varphi^{-1}_n\varphi_m(0_M)) < \varepsilon$. 
	
	We define $F_n\coloneqq \cK^{\varphi_n} \bullet \bfk_{M \times [0,\infty)} \in \cD(M)$.
	By \cref{proposition:spectral_norm_distance}, 
	\begin{equation}
		\gamma (\varphi^{-1}_n\varphi_m(0_M))=
		\overline{d_{\cD(M)}}(\bfk_{M\times [0,\infty)}, \cK^{\varphi^{-1}_n} \bullet \cK^{\varphi_m} \bullet 
		\bfk_{M\times [0,\infty)})=
		\overline{d_{\cD(M)}}(F_n, F_m).
	\end{equation}
	Hence, the sequence $(F_n)_{n \in \bN}$ is a Cauchy sequence with respect to $\overline{d_{\cD(M)}}$. 
	This implies that there exists a sequence $(c_n)_{n \in \bN}$ of real numbers such that $(T_{c_n}F_n)_{n \in \bN}$ is Cauchy with respect to $d_{\cD(M)}$. 
	Thus, there exists a limit object $F_\infty$ of $(T_{c_n}F_n)_{n \in \bN}$ by \cref{corollary:limit_Tamarkin}. 
	By construction, we obtain $d_{\cD(M)}(\bfk_{M\times [0,\infty)}, F_\infty) < \infty$ and the desired microsupport estimate.  
	
	\noindent (ii) By construction, we find that 
	\begin{equation}
		\Spec(F_\infty) = \lim_{n \to \infty} \Spec(T_{c_n}F_n) = \lim_{n \to \infty} T_{-c_n} \Spec(\varphi_n, 0_M). 
	\end{equation}
	Thus, applying \cref{theorem:sheaf_spectral_invariant} to $F_\infty$, we obtain the result.
\end{proof}

Note that the number of $\Spec(\varphi_\infty,0_M)$ is the same as that of \cite{buhovsky2019arnold}.

\section{Arnold-type principle for Hausdorff limits of Legendrians}\label{section:legendrian}

In this section, we briefly discuss how to prove a Legendrian analogue of \cref{theorem:ineq_hameo} (cf.\ \cite[Thm.~1.5]{buhovsky2019arnold}) by a sheaf-theoretic method. 
Again, we assume that $M$ is compact.

Denote by $J^1M=T^*M \times \bR$ the 1-jet bundle. 
For a compact Legendrian submanifold $L$ of $J^1M$, we define a conic Lagrangian submanifold $c(L)$ of $T^*(M \times \bR_t)$ by 
\begin{equation}
	c(L) \coloneqq \{ (x,t;\xi, \tau) \mid \tau >0, (x,\xi/\tau,t) \in L \}.
\end{equation}
Let $L$ be a Legendrian submanifold without Reeb chords. 
Then by the results in \cite[Part~XII]{Gu23}, we can construct $F_L \in \SD(\bfk_{M \times \bR_t})$ such that $\rMS(F_L)=c(L)$.

Consider a sequence $(L_n)_{n \in \bN}$ of compact Legendrian submanifolds without Reeb chords of $J^1M$. 
Assume that $(L_n)_{n \in \bN}$ converges to a compact subset $L_\infty$ of $J^1M$ with respect to the Hausdorff distance.  
For each $L_n$, we can construct $F_n \coloneqq F_{L_n} \in \SD(\bfk_{M \times \bR_t})$ as above.
Following \cite{guillermou2013gromov} (see also \cite[Part~VII]{Gu23}), we define $F_\infty \in \SD(\bfk_{M \times \bR_t})$ by the exact triangle
\begin{equation}
	\bigoplus_{n \in \bN} F_n \lto \prod_{n \in \bN} F_n \lto F_\infty \toone.  
\end{equation}
Then we find $F_\infty \in \cD(M)$ and get $Q_\infty^*(F_\infty) \simeq C \otimes H^*(M;\bfk)$, where $C \coloneqq \Coker(\bigoplus_{n \in \bN} \bfk \to \prod_{n \in \bN} \bfk)$.
Applying \cref{lem:ms-prod-sum} and arguing as in the proof of \cref{lemma:ms_hocolim}, we have 
\begin{equation}\label{eq:ms_limit_legendrian}
	\rMS(F_\infty) \subset c(L_\infty) = \lim_{n \to \infty} c(L_n) \subset T^*(M \times \bR_t).
\end{equation}
Let $t \colon M \times \bR_t \to \bR_t$ and $q_\bR \colon J^1M=T^*M \times \bR_t \to \bR_t$ be the projections. 
Then we find that 
\begin{equation}
	\MS(F_\infty) \cap \Gamma_{dt} 
	\subset 
	\{ (x,t;0,1) \mid (x,0,t) \in L_\infty \}
	= (L_\infty \cap (0_M \times \bR_t)) \times \{1\}
\end{equation}
and hence
\begin{equation}
	-\Spec(F_\infty) \subset t \pi(\MS(F_\infty) \cap \Gamma_{dt}) \subset q_\bR(L_\infty \cap (0_M \times \bR_t)).
\end{equation}

\begin{proposition}
	In the situation as above, if $\Spec(F_\infty) \le \cl(M)$, then $L_\infty \cap (0_M \times \bR_t)$ is cohomologically non-trivial in $M \times \bR_t$. 
	In particular, if $\# q_\bR(L_\infty \cap (0_M \times \bR_t)) \le \cl(M)$, then $L_\infty \cap (0_M \times \bR_t)$ is cohomologically non-trivial in $M \times \bR_t$, hence it is infinite.
\end{proposition}

\begin{proof}
	By applying \cref{theorem:sheaf_spectral_invariant} to the object $F_\infty \in \cD(M)$, we obtain the result by \eqref{eq:ms_limit_legendrian} and $Q_\infty^*(F_\infty) \simeq C \otimes H^*(M;\bfk)$.
\end{proof}

\appendix

\section{Hamiltonian stability with support conditions}\label{section:support_condition}

In this appendix, we prove an estimate by the oscillation norm $\|H\|_{\mathrm{osc},A}$ of $H$ restricted to a non-empty closed subset $A$, in the context of the sheaf-theoretic energy estimate.

\subsection{Sheaf quantization of 2-parameter Hamiltonian isotopies}\label{subsec:SQ_ham}

We will use the sheaf quantization of a 2-parameter Hamiltonian isotopy. 
For that purpose, we first state the main result of \cite{GKS} in a general form. 
Let $N$ be a connected non-empty manifold and $W$ a contractible open subset of $\bR^n$ with the coordinate system $(w_1,\dots,w_n)$ containing $0$. 
Let us consider $\psi =(\psi_w)_{w \in W} \colon \rT N \times W \to \rT N$ be a homogeneous Hamiltonian isotopy, that is, a $C^\infty$-map satisfying (1)~$\psi_w$ is homogeneous symplectic isomorphism for each $w \in W$ and (2)~$\psi_0=\id_{\rT N}$.
We can define a vector-valued homogeneous function $h \colon \rT N \times W \to \bR^n$ by $h=(h_1,\dots,h_n)$ with 
\begin{equation}
	\frac{\partial \psi_w}{\partial w_i} \circ \psi_{w}^{-1} = X_{h_i(\bullet,w)},
\end{equation}
where $X_{h_i(\bullet,w)}$ is the Hamiltonian vector field of the function $h_i(\bullet,w) \colon \rT N \to \bR$.
By using the function $h$, we define a conic Lagrangian submanifold $\Lambda_\psi$ of $\rT (N^2 \times W)$ by 
\begin{equation}
	\Lambda_\psi \coloneqq \lc \lb \psi_w(y;\eta), (y;-\eta), (u;-h(\psi_w(y;\eta),w)) \rb \relmid (y;\eta) \in \rT N, w \in W \rc
\end{equation}
The main theorem of \cite{GKS} is the following.

\begin{theorem}[{\cite[Thm.~3.7 and Rem.~3.9]{GKS}}]\label{thm:GKSmain_general}
	Let $\psi \colon \rT N \times W \to \rT N$ be a homogeneous Hamiltonian isotopy and set $\Lambda_\psi$ as above. 
	Then there exists a unique simple object $\tl{K} \in \SD(\bfk_{N^2 \times W})$ such that $\rMS(\tl{K})=\Lambda_{\psi}$ and $\tl{K}|_{N^2 \times \{0\}} \simeq \bfk_{\Delta_{N}}$.
\end{theorem}

For a non-homogeneous compactly supported Hamiltonian isotopy, we can associate a sheaf by homogenizing the isotopy.
In the 1-parameter case, it is done as in \cref{definition:homogeneous_ham}.
Below we will explain how to homogenize a 2-parameter Hamiltonian isotopy.

Let $(G_{s',s})_{(s',s) \in I^2}$ be a 2-parameter family of compactly supported smooth functions on $T^*M$. 
A 2-parameter family of diffeomorphisms $(\phi_{s',s})_{(s',s)\in I^2}$ is determined by $\phi_{s',0}=\id_{T^*M} $ and $\frac{\partial \phi_{s',s} }{\partial s}\circ \phi_{s',s}^{-1}=X_{G_{s',s}}$, where $X_{G_{s',s}}$ is the Hamiltonian vector field corresponding to the function $G_{s',s}$. 
We set $\widehat{G}_{s',s}(x,t;\xi,\tau) \coloneqq \tau G_{s',s}(x;\xi/\tau)$ and define 
a 2-parameter homogeneous Hamiltonian isotopy $\wh{\phi}=(\widehat{\phi}_{s',s})_{s',s}$ by 
\begin{equation}
	\begin{cases}
		\widehat{\phi}_{s',0}=\id_{\rT(M \times \bR_t)}, \\
		\frac{\partial \widehat{\phi}_{s',s}}{\partial s} \circ \widehat{\phi}_{s',s}^{-1} = X_{\widehat{G}_{s',s}}.
	\end{cases}
\end{equation}
Then, we have
\begin{equation}
	\begin{aligned}
		\Lambda_{\wh{\phi}} = & \left\{ \lb \widehat{\phi}_{s',s}(y;\eta), (y;-\eta), (s'; -\widehat{F}_{s',s}(\widehat{\phi}_{s',s}(y;\eta))), (s;- \widehat{G}_{s',s}(\widehat{\phi}_{s',s}(y;\eta))) \rb \relmid \right. \\ 
		& \hspace{10pt} \left. (y;\eta) \in \rT(M \times \bR), s',s \in I \right\},
	\end{aligned}
\end{equation}
where the 2-parameter family of homogeneous functions $(\widehat{F}_{s',s})_{s',s}$ is determined by $\frac{\partial \widehat{\phi}_{s',s}}{\partial s'} \circ \widehat{\phi}_{s',s}^{-1}=X_{\widehat{F}_{s',s}}$.
By the construction of $\widehat{\phi}$, there exists a 2-parameter family of timewise compactly supported functions $(F_{s',s})_{(s',s) \in I^2}$ satisfying $\widehat{F}_{s',s}(x,t;\xi,\tau)=\tau F_{s',s}(x;\xi/\tau)$ and $\frac{\partial \phi_{s',s} }{\partial s'}\circ \phi_{s',s}^{-1}=X_{F_{s',s}}$.
A calculation in \cite{polterovich2012geometry} or \cite{Oh02normalization} (see also \cite{banyaga1978structure}) shows that 
\begin{equation}\label{eq:another_function}
	\frac{\partial F_{s',s}}{\partial s}=\frac{\partial G_{s',s}}{\partial s'}-\{F_{s',s},G_{s',s}\},
\end{equation}
where $\{-,-\}$ is the Poisson bracket. 
In this case, we can apply \cref{thm:GKSmain_general} to the homogeneous Hamiltonian isotopy $\wh{\phi}$ and obtain a simple object $\tl{K}$ satisfying $\rMS(\tl{K})=\Lambda_{\wh{\phi}}$.
By using the map $q \colon (M \times \bR)^2 \times I^2 \to M^2 \times \bR_t \times I^2, (x_1,t_1,x_2,t_2,s',s) \mapsto (x_1,x_2,t_1-t_2,s',s)$, we also obtain an equivalence similar to \eqref{eq:equivalence_GKS}. 
Hence, we can define $K \in \SD(\bfk_{M^2 \times \bR \times I^2})$ by the condition $\rMS(K)=q_d q_\pi^{-1} (\Lambda_{\wh{\phi}})$ and $\cK \coloneqq P_l(K) \in \cD(M^2 \times I^2)$, which we call the \emph{sheaf quantization} of $(\phi_{s',s})_{(s',s)\in I^2}$.

\subsection{Statement and proof}

For a closed subset $A$ of $T^*M$, we define a full subcategory $\cD_A(M)$ of $\cD(M)$ by 
\begin{equation}
	\cD_A(M) \coloneqq \{ F \in \cD(M) \mid \MS(F) \cap \{ \tau >0\} \subset \rho_t^{-1}(A) \},
\end{equation}
where $\rho_t \colon T^*M \times T^*_{\tau >0}\bR_t \to T^*M, (x,t;\xi,\tau) \mapsto (x;\xi/\tau)$.

Let $\cK^H \in \cD(M^2 \times I)$ be the sheaf quantization associated with a timewise compactly supported function $H \colon T^*M \times I \to \bR$ and $F \in \cD_A(M)$ with $A$ being a closed subset of $T^*M$.
Then we get $\cK^H \bullet F \in \cD(M \times I)$ and find that 
\begin{equation}
	\cK^H_s \bullet F \simeq (\cK^H \bullet F)|_{M \times \{s\} \times \bR_t} \in \cD_{\phi^H_s(A)}(M) \quad \text{for any $s \in I$}.
\end{equation}
We shall estimate the distance between $F \in \cD_A(M)$ and $\cK^H_1 \bullet F \in \cD_{\phi^H_1(A)}(M)$ up to translation.
See also \cref{remark:weak} for a more straightforward but weaker case. 

\begin{theorem}\label{theorem:metric_support}
	Let $A$ be a non-empty closed subset of $T^*M$ and $F \in \cD_A(M)$.
	Moreover, let $H \colon T^*M \times I \to \bR$ be a timewise compactly supported function. 
	Then for a continuous function $f\colon I\to \bR$, one has
	\begin{equation}\label{eq:bound}
		\begin{aligned}
			& d_{\cD(M)}(F,T_{-c} \cK^H_1 \bullet F) \\
			\le {} & \int_0^1 \lb \max \left\{ \max_{p \in \phi^H_s(A)}H_s(p), f(s) \right\} -\min \left\{ \min_{p \in \phi^H_s(A)}H_s(p), f(s)\right\} \rb ds
		\end{aligned}	
	\end{equation}
	where $c=\int_0^1 f(s) ds$. 
\end{theorem}

\begin{remark}\label{remark:metric_support}
	If we take $f\equiv 0$, we obtain 
	\begin{equation}
		\begin{aligned}
			& d_{\cD(M)}(F, \cK^H_1 \bullet F) \\
			\le {} & \int_0^1 \lb \max \left\{ \max_{p \in \phi^H_s(A)}H_s(p), 0 \right\} -\min \left\{ \min_{p \in \phi^H_s(A)}H_s(p), 0\right\} \rb ds. 
		\end{aligned}
	\end{equation}
	
	Let $c \in \bR$ be a real number satisfying 
	\begin{equation}
		\int_0^1 \min_{p \in \phi^H_s(A)}H_s(p) ds\le c \le \int_0^1 \max_{p \in \phi^H_s(A)}H_s(p) ds. 
	\end{equation}
	Then we can take $f$ such that $c=\int_0^1 f(s) ds$ and 
	\begin{equation}
		\min_{p \in \phi^H_s(A)}H_s(p) \le f(s)\le \max_{p \in \phi^H_s(A)}H_s(p)
	\end{equation}
	for any $s \in I$. 
	Hence, by \cref{theorem:metric_support}, we get
	\begin{equation}
		d_{\cD(M)}(F, T_{-c} \cK^H_1 \bullet F)
		\le \int_0^1 \lb \max_{p \in \phi^H_s(A)}H_s(p) - \min_{p \in \phi^H_s(A)}H_s(p) \rb ds.
	\end{equation}
\end{remark}

For simplicity, we introduce a symbol for the right-hand side of \eqref{eq:bound}. 
For a function $H \colon T^*M \times I \to \bR$, a function $f \colon I \to \bR$, and a non-empty closed subset $A$ of $T^*M$, we set 
\begin{equation}
	B(H,f,A) 
	\coloneqq 
	\int_0^1 \lb \max \left\{ \max_{p \in \phi^H_s(A)}H_s(p), f(s) \right\} -\min \left\{ \min_{p \in \phi^H_s(A)}H_s(p), f(s) \right\} \rb ds. 
\end{equation}

\begin{proof}[Proof of \cref{theorem:metric_support}]
	Let $\varepsilon >0$.
	We can take a smooth family $(\rho_{a,b}\colon \bR\to \bR)_{a,b}$ of smooth functions parametrized by $a,b\in \bR$ with $a\le b$ such that 
	\begin{enumerate}
		\renewcommand{\labelenumi}{$\mathrm{(\arabic{enumi})}$}
		\item $\rho_{a,b}(y)=y$ on a neighborhood of $[a,b]$,
		\item $a-\varepsilon \le \inf_y \rho_{a,b}(y) < \sup_y \rho_{a,b}(y) \le b+\varepsilon$. 
	\end{enumerate}
	Recall that $I$ denotes an open interval containing the closed interval $[0,1]$.
	We take smooth functions $M, m \colon I\to \bR$ satisfying 
	\begin{equation}
		\max_{p \in \phi^H_s(A)}H_s(p) +\frac{\varepsilon}{2} \le M(s) \le \max_{p \in \phi^H_s(A)}H_s(p)  +\varepsilon
	\end{equation}
	and
	\begin{equation}
		\min_{p \in \phi^H_s(A)}H_s(p) -\varepsilon \le m(s) \le \min_{p \in \phi^H_s(A)}H_s(p)-\frac{\varepsilon}{2}.
	\end{equation}
	Fix $R>0$ sufficiently large so that $R> \max_{p,s}H_s(p) -\min_{p,s}H_s(p) +2\varepsilon$.
	Define $a(s', s)\coloneqq m(s)-Rs', b(s',s)\coloneqq M(s)+Rs'$ for $(s',s)\in I^2$. 
	We may assume that $I\subset (-\frac{\varepsilon}{2R},+\infty)$ by taking $I$ smaller if necessary, and hence that 
	\begin{equation}
		a(s', s) \le \min_{p \in \phi^H_s(A)}H_s(p) \le \max_{p \in \phi^H_s(A)}H_s(p) \le b(s',s)
	\end{equation}
	for all $(s',s)\in I^2$. 
	Take a smooth function $\tilde{f}\colon I\to \bR$ such that $\|\tilde{f}-f\|_{C^0} \le \varepsilon$. 
	By shrinking $I$, we may assume that $\bigcup_{s} \supp(H_s)$ is relatively compact.
	Then we can also take a compactly supported smooth cut-off function $\chi \colon T^*M\to [0,1]$ such that $\chi\equiv 1$ on a neighborhood of $\bigcup_{s} \supp(H_s)$.
	Using these functions, we define a function $G=(G_{s',s})_{s',s \in I^2} \colon T^*M \times I^2 \to \bR$ by 
	\begin{equation}
		G_{s',s}\coloneqq \left(\rho_{a(s',s),b(s',s)}\circ H_s - (1-s') \tilde{f}(s)\right) \chi.
	\end{equation}
	A 2-parameter family $(\phi_{s',s})_{(s',s)\in I^2}$ of Hamiltonian diffeomorphisms is determined by $\phi_{s',0}=\id_{T^*M} $ and $\frac{\partial \phi_{s',s} }{\partial s}\circ \phi_{s',s}^{-1}=X_{G_{s',s}}$, where $X_{G_{s',s}}$ is the Hamiltonian vector field corresponding to the function $G_{s',s}$. 
	Note that $G_{1,s}=H_s$ and $\phi_{s',s}$ is independent of $s'$ on a neighborhood $U$ of $A$. 
	Moreover, we have 
	\begin{equation}
		\begin{aligned}
			& \int_0^1 \lb \max_{p \in T^*M} G_{0,s}(p) - \min_{p \in T^*M} G_{0,s}(p) \rb ds \\
			\le {} & 
			\int_0^1 \lb \max_{p \in T^*M} \left( \rho_{m(s),M(s)} \circ H_s(p) - \tilde{f}(s) \right)\chi(p) - \min_{p \in T^*M} \left( \rho_{m(s),M(s)} \circ H_s(p) - \tilde{f}(s) \right) \chi(p) \rb ds
			\\ 
			\le {} & 
			\int_0^1 \lb \max \left\{ \max_{p \in \phi^H_s(A)} \left( H_s(p) - \tilde{f}(s) \right), 0 \right\} -\min \left\{ \min_{p \in \phi^H_s(A)} \left( H_s(p) - \tilde{f}(s) \right), 0\right\} \rb ds +2\varepsilon \\
			= {} & 
			\int_0^1 \lb \max \left\{ \max_{p \in \phi^H_s(A)} H_s(p), \tilde{f}(s) \right\} -\min \left\{ \min_{p \in \phi^H_s(A)} H_s(p), \tilde{f}(s) \right\} \rb ds +2\varepsilon \\
			\le {} &
			B(H,f,A) +4\varepsilon.
		\end{aligned}
	\end{equation}
	For $s' \in I$, we set $G_{s'} \coloneqq G_{s',\bullet} \colon T^*M \times I \to \bR$.
	Then, by \cref{theorem:SQ_inequality} and the natural inequality for the distance with respect to functorial operations (see \eqref{eqn:bullet_distance}), we obtain 
	\begin{equation}
		d_{\cD(M)}(F, \cK^{G_0}_1 \bullet F) = d_{\cD(M)}(\cK^0_1 \bullet F, \cK^{G_0}_1 \bullet F) 
		\le 
		B(H,f,A) +4\varepsilon.
	\end{equation}
	
	We set $\tilde{c}(s)\coloneqq  \int_0^s \tilde{f}(t) dt$ and claim that $\cK^{G_0}_1 \bullet F \simeq T_{-\tilde{c}(1)}\cK^{G_1}_1 \bullet F$. 
	By the result recalled in \cref{subsec:SQ_ham}, we can construct the sheaf quantization $\cK \in \cD(M^2 \times I^2)$ of the 2-parameter family of diffeomorphisms $(\phi_{s',s})_{s',s}$. 
	We shall use the same notation as in \cref{subsec:SQ_ham}.
	Then, $F_{s',0}=0$ and $F_{\bullet,s}|_{\phi_{1,s}(U)\times I}\colon \phi_{1,s}(U)\times I\to \bR, (p,s')\mapsto F_{s',s}(p)$ is locally constant for each $s$. 
	By \eqref{eq:another_function}, we find that $\frac{\partial F_{s',s}}{\partial s}=\frac{\partial G_{s',s}}{\partial s'}=\tilde{f}(s)$ on $\bigcup_s \phi_{1,s}(U)\times I\times \{s\}$ and that $F_{s',s} = \int_0^s \tilde{f}(t) dt=\tilde{c}(s)$ there. 
	We define $\cH \coloneqq \cK \bullet F \in \cD(M \times I^2)$. 
	Then, by the microsupport estimate, we have 
	\begin{equation}
		\begin{aligned}
			\MS(\cH) 
			\subset 
			& \left\{ \lb \widehat{\phi}_{s',s}(x,t;\xi,\tau), (s';-\tau \tilde{c}(s)), (s;-\tau G_{s',s}(\phi_{s',s}(x;\xi/\tau))) \rb \; \middle| \right. \\ 
			& \hspace{10pt} \left. (x,t;\xi,\tau) \in \rMS(F), s',s \in I \right\} 
			\cup 0_{M \times \bR \times I^2}.
		\end{aligned}  
	\end{equation}
	Hence, $M \times \bR \times I \times \{1\}$ is non-characteristic for $\cH$ and we get 
	\begin{equation}
		\begin{aligned}
			\MS(\cH|_{M \times \bR \times I \times \{1\}})
			\subset & 
			\left\{ \lb \widehat{\phi}_{s',1}(x,t;\xi,\tau), (s';- \tilde{c}(1) \tau) \rb \; \middle| \right. \\
			& \quad \left. (x,t;\xi,\tau) \in \rMS(F), s' \in I \right\} 
			\cup 0_{M \times \bR \times I}.
		\end{aligned}       
	\end{equation}
	Define a diffeomorphism $\varphi \colon M \times \bR \times I \simto M \times \bR \times I, (x,t,s') \mapsto (x,t-\tilde{c}(1)s', s')$. 
	Then we have $\MS(\varphi_* \cH|_{M \times \bR \times I \times \{1\}}) \subset T^*(M \times \bR) \times 0_{I}$, which shows $\varphi_* \cH|_{M \times \bR \times I \times \{1\}}$ is the pull-back of a sheaf on $M \times \bR$ by \cite[Prop.~5.4.5]{KS90}.
	In particular, 
	\begin{equation}
		\begin{aligned}
			\cK^{G_0}_1 \bullet F 
			& = 
			\cH|_{M \times \bR \times \{0\} \times \{1\}} \\
			& \simeq 
			(\varphi_* \cH|_{M \times \bR \times I \times \{1\}})|_{\{ s'=0\}} \\
			& \simeq 
			(\varphi_* \cH|_{M \times \bR \times I \times \{1\}})|_{\{ s'=1\}} \\
			& \simeq 
			T_{-\tilde{c}(1)}\cH|_{M \times \bR \times \{1\} \times \{1\}} 
			=
			T_{-\tilde{c}(1)} \cK^{G_1}_1 \bullet F. 
		\end{aligned}
	\end{equation}
	
	Since $|c-\tilde{c}(1)| \le \varepsilon$, we have $d_{\cD(M)}(T_{-c}\cK^H_1 \bullet F, T_{-\tilde{c}(1)}\cK^H_1 \bullet F) \le \varepsilon$.
	Combining the result above and noticing $G_1=H$, we obtain 
	\begin{equation}
		\begin{aligned}
			d_{\cD(M)}(F, T_{-c} \cK^H_1 \bullet F)
			& \le 
			d_{\cD(M)}(F, T_{-\tilde{c}(1)} \cK^H_1 \bullet F) + \varepsilon \\
			& = 
			d_{\cD(M)}(F, \cK^{G_0}_1 \bullet F) + \varepsilon \\
			& \le 
			B(H,f,A) +5\varepsilon.
		\end{aligned}
	\end{equation}
	Since $\varepsilon>0$ is arbitrary, this completes the proof.
\end{proof}

\begin{remark}\label{remark:weak}
	Under the same assumption as in \cref{theorem:metric_support}, we can prove the weaker result 
	\begin{equation}
		d_{\mathrm{w\text{-}isom}}(F,T_{-c} \cK^H_1 \bullet F) \le B(H,f,A) 
	\end{equation}
	more straightforwardly, without the 2-parameter family, as follows. 
	Here $d_{\mathrm{w\text{-}isom}}$ denotes the pseudo-distance on $\cD(M)$ defined by 
	\begin{equation}
		d_{\mathrm{w\text{-}isom}}(F,G) \coloneqq \inf \lc a+b \relmid \text{$(F,G)$ is weakly $(a,b)$-isomorphic} \rc. 
	\end{equation}
	
	We set $\cH \coloneqq \cK^H \bullet F \in \cD(M \times I)$. 
	Then we have $\cH|_{M \times \bR_t \times \{0\}} \simeq F$ 
	and $\cH|_{M \times \bR_t \times \{1\}} \simeq \cK^H_1 \bullet F$.
	Moreover, by the microsupport estimate, we find that 
	\begin{equation}
		\MS(\cH) 
		\subset T^*M \times \left\{ (t,s;\tau,\sigma) \;\middle|\;  -\max_{p \in \phi^H_s(A)} H_s(p) \cdot \tau \le \sigma \le -\min_{p \in \phi^H_s(A)} H_s(p) \cdot \tau \right\}.
	\end{equation}
	Let $\varepsilon>0$ and take a smooth function $\tilde{f}\colon I\to \bR$ such that $\|\tilde{f}-f\|_{C^0} \le \varepsilon$. 
	We define a function $\tilde{c} \colon I \to \bR$ by $\tilde{c}(s) \coloneqq \int_0^s \tilde{f}(s')ds'$ and    
	a function $\varphi \colon M \times \bR_t \times I  \to M \times \bR_t \times I$ by $\varphi(x,t,s) \coloneqq (x,t-\tilde{c}(s),s)$.
	Then we have $\varphi_* \cH|_{M \times \bR_t \times \{0\}} \simeq F$, $\varphi_* \cH|_{M \times \bR_t \times \{1\}} \simeq T_{-\tilde{c}(1)}\cK^H_1 \bullet F$, and     
	\begin{equation}
    \begin{aligned}
        & \MS(\varphi_* \cH) \\
        & \quad \subset 
        T^*M \times \left\{ (t,s;\tau,\sigma) \;\middle|\; - \lb \max_{p \in \phi^H_s(A)} H_s(p) -\tilde{f}(s) \rb \cdot \tau \le \sigma \right. \\ 
        & \hspace{125pt} \left. \le - \lb \min_{p \in \phi^H_s(A)} H_s(p) - \tilde{f}(s) \rb \cdot \tau \right\}.
    \end{aligned} 
	\end{equation}
	Note that we may have $\max_{p \in \phi^H_s(A)} H_s(p) -\tilde{f}(s) < 0$ and $\min_{p \in \phi^H_s(A)} H_s(p) - \tilde{f}(s) >0$ in general.
	By applying \cref{proposition:abisomhtpy}, we obtain 
	\begin{equation}
		\begin{aligned}
			& d_{\mathrm{w\text{-}isom}}(F,T_{-\tilde{c}(1)}\cK^H_1 \bullet F) \\
			= {} &  
			d_{\mathrm{w\text{-}isom}}(\varphi_* \cH|_{M \times \bR_t \times \{0\}}, \varphi_* \cH|_{M \times \bR_t \times \{1\}}) \\
			\le {} & 
			\int_0^1 \lb \max \left\{ \max_{p \in \phi^H_s(A)} \left( H_s(p) - \tilde{f}(s) \right), 0 \right\} -\min \left\{ \min_{p \in \phi^H_s(A)} \left( H_s(p) - \tilde{f}(s) \right), 0\right\} \rb ds \\
			= {} &
			\int_0^1 \lb \max \left\{ \max_{p \in \phi^H_s(A)} H_s(p), \tilde{f}(s) \right\} -\min \left\{ \min_{p \in \phi^H_s(A)} H_s(p), \tilde{f}(s) \right\} \rb ds \\
			\le {} & 
			\int_0^1 \lb \max \left\{ \max_{p \in \phi^H_s(A)}H_s(p), f(s) \right\} -\min \left\{ \min_{p \in \phi^H_s(A)}H_s(p), f(s)\right\} \rb ds + 2\varepsilon.
		\end{aligned}
	\end{equation}
	Hence, we have 
	\begin{equation}
		\begin{aligned}
			& d_{\mathrm{w\text{-}isom}}(F, T_{-c}\cK^H_1 \bullet F) 
			\le d_{\mathrm{w\text{-}isom}}(F, T_{-\tilde{c}(1)} \cK^H_1 \bullet F) + \varepsilon \\
			\le {} &
			\int_0^1 \lb \max \left\{ \max_{p \in \phi^H_s(A)}H_s(p), f(s) \right\} -\min \left\{ \min_{p \in \phi^H_s(A)}H_s(p), f(s)\right\} \rb ds + 3\varepsilon, 
		\end{aligned}
	\end{equation}
	which completes the proof.    
\end{remark}

For a timewise compactly supported function $H \colon T^*M \times I \to \bR$ and a non-empty closed subset $A$ of $T^*M$, we set 
\begin{equation}
	\|H\|_{\mathrm{osc},A} \coloneqq \int_0^1 \left(\max_{p \in A} H_s(p) - \min_{p \in A} H_s(p) \right) ds.
\end{equation}

\begin{proposition}
	Let $A$ be a non-empty closed subset of $T^*M$ and $F \in \cD_A(M)$.
	Moreover, let $H \colon T^*M \times I \to \bR$ be a timewise compactly supported function. 
	Then there exists $c \in \bR$ such that 
	\begin{equation}
		d_{\cD(M)}(F,T_{-c}\cK^H_1 \bullet F)
		\le \|H\|_{\mathrm{osc},A}.
	\end{equation}
\end{proposition}

\begin{proof}
	Using the technique in the proof of \cite[Theorem~1.3]{usher2015observations}, one can construct a function $H'$ such that $\phi^H_1=\phi^{H'}_1$ and 
	\begin{equation}
    \begin{aligned}
        & \int_0^1 \lb \max_{p \in A}H_s(p) - \min_{p \in A}H_s(p) \rb ds \\
        ={} & \int_0^1 \lb \max_{p \in \phi^{H'}_s(A)}H'_s(p) - \min_{p \in \phi^{H'}_s(A)}H'_s(p) \rb ds.
    \end{aligned}
	\end{equation}
	By \cref{proposition:dependence_time_one}, $\phi^H_1=\phi^{H'}_1$ implies $\cK^H_1 \simeq \cK^{H'}_1$. 
	Hence, the result follows from \cref{theorem:metric_support} (see also \cref{remark:metric_support}).
\end{proof}

\printbibliography

\noindent Tomohiro Asano: 
Research Institute for Mathematical Sciences, Kyoto University, \linebreak Kitashirakawa-Oiwake-Cho, Sakyo-ku, 606-8502, Kyoto, Japan.

\noindent \textit{E-mail address}: \texttt{tasano[at]kurims.kyoto-u.ac.jp}, \texttt{tomoh.asano[at]gmail.com}

\medskip

\noindent Yuichi Ike:
Institute of Mathematics for Industry, Kyushu University, 744 Motooka Nishi-ku, 819-0395, Fukuoka, Japan.

\noindent
\textit{E-mail address}: \texttt{ike[at]imi.kyushu-u.ac.jp}, \texttt{yuichi.ike.1990[at]gmail.com}

\end{document}